\pdfoutput=1
\RequirePackage{ifpdf}
\ifpdf 
\documentclass[pdftex]{sigma}
\else
\documentclass{sigma}
\fi
\RequirePackage{color}

\DeclareMathOperator{\tr}{T}

\newtheorem{problem}{Problem}
\newtheorem{reformulation}{Reformulation}

{ \theoremstyle{definition}
\newtheorem{algorithm}{Algorithm}
\newtheorem{notation}{Notation}
}

\def\blue{\textcolor{blue} }
\def\red{\textcolor{red} }

\definecolor{brown}{cmyk}{0, 0.8, 1, 0.6}
\definecolor{ggr}{rgb}{.20,.60,.22}
\def\brown{\textcolor{brown} }
\def\viol{\textcolor{darkgreen} }
\newcommand\pd[2]{\frac{\partial #1}{\partial #2}}
\newcommand\jp[1]{{\bf p}^{(#1)}}

\newcommand\jy[1]{{ y}^{(#1)}}
\newcommand{\affp}{{\mathcal A(2)}\times{\mathcal A(3)}}
\newcommand{\pta}{{{\mathcal {PGL}}(3)}\times{\mathcal A(3)}}
\newcommand{\fp}{{\mathcal {CP}}}
\newcommand{\ap}{{\mathcal {PP}}}
\newcommand{\stpf}{{ P_{\mathcal C}^0}}
\newcommand{\stpa}{{ P_{\mathcal P}^0}}

\begin{document}

\allowdisplaybreaks

\renewcommand{\thefootnote}{$\star$}

\renewcommand{\PaperNumber}{023}

\FirstPageHeading

\ShortArticleName{Object-Image Correspondence for Algebraic Curves under Projections}

\ArticleName{Object-Image Correspondence\\
for Algebraic Curves under Projections\footnote{This paper is a~contribution to the Special Issue
``Symmetries of Differential Equations: Frames, Invariants and~Applications''.
The full collection is available at
\href{http://www.emis.de/journals/SIGMA/SDE2012.html}{http://www.emis.de/journals/SIGMA/SDE2012.html}}}

\Author{Joseph M.~BURDIS, Irina A.~KOGAN and Hoon HONG}
\AuthorNameForHeading{J.M.~Burdis, I.A.~Kogan and H.~Hong}
\Address{North Carolina State University, USA}
\Email{\href{joe.burdis@gmail.com}{joe.burdis@gmail.com}, \href{iakogan@ncsu.edu}{iakogan@ncsu.edu}, \href{hong@ncsu.edu}{hong@ncsu.edu}}
\URLaddress{\url{http://www.linkedin.com/in/josephburdis},
\\
\hspace*{10.5mm}\url{http://www.math.ncsu.edu/~iakogan/}, \url{http://www.math.ncsu.edu/~hong/}}

\ArticleDates{Received October 01, 2012, in final form March 01, 2013; Published online March 14, 2013}

\Abstract{We present a~novel algorithm for deciding whether a~given planar curve is an image of
a~given spatial curve, obtained by a~central or a~parallel projection with unknown parameters.
The motivation comes from the problem of establishing a~correspondence between an object and an
image, taken by a~camera with unknown position and parameters.
A straightforward approach to this problem consists of setting up a~system of conditions on the
projection parameters and then checking whether or not this system has a~solution.
The computational advantage of the algorithm presented here, in comparison to algorithms based on
the straightforward approach, lies in a~significant reduction of a~number of real parameters that
need to be eliminated in order to establish existence or non-existence of a~projection that maps
a~given spatial curve to a~given planar curve.
Our algorithm is based on projection criteria that reduce the projection problem to a~certain
modification of the equivalence problem of planar curves under affine and projective
transformations.
To solve the latter problem we make an algebraic adaptation of signature construction that has been
used to solve the equivalence problems for smooth curves.
We introduce a~notion of a~classifying set of rational differential invariants and produce
explicit formulas for such invariants for the actions of the projective and the affine groups
on the plane.}


\Classification{14H50; 14Q05; 14L24; 53A55; 68T45}

\rightline{\it The paper is dedicated to Peter Olver's 60th birthday.}

\renewcommand{\thefootnote}{\arabic{footnote}}
\setcounter{footnote}{0}

\section{Introduction}

Identifying an object in three-dimensional space with its planar image is
a~fundamental problem in computer vision.
In particular, given a~database of images (medical images, aerial photographs, human photographs),
one would like to have an algorithm to match a~given object in 3D with an image in the database,
even though a~position of the camera and its parameters may be unknown.
Since the defining features of many objects can be represented by curves, obtaining a~solution
for the identification problem for curves is essential.

\begin{figure}[t]
\centering \includegraphics[angle=270,width=55mm]{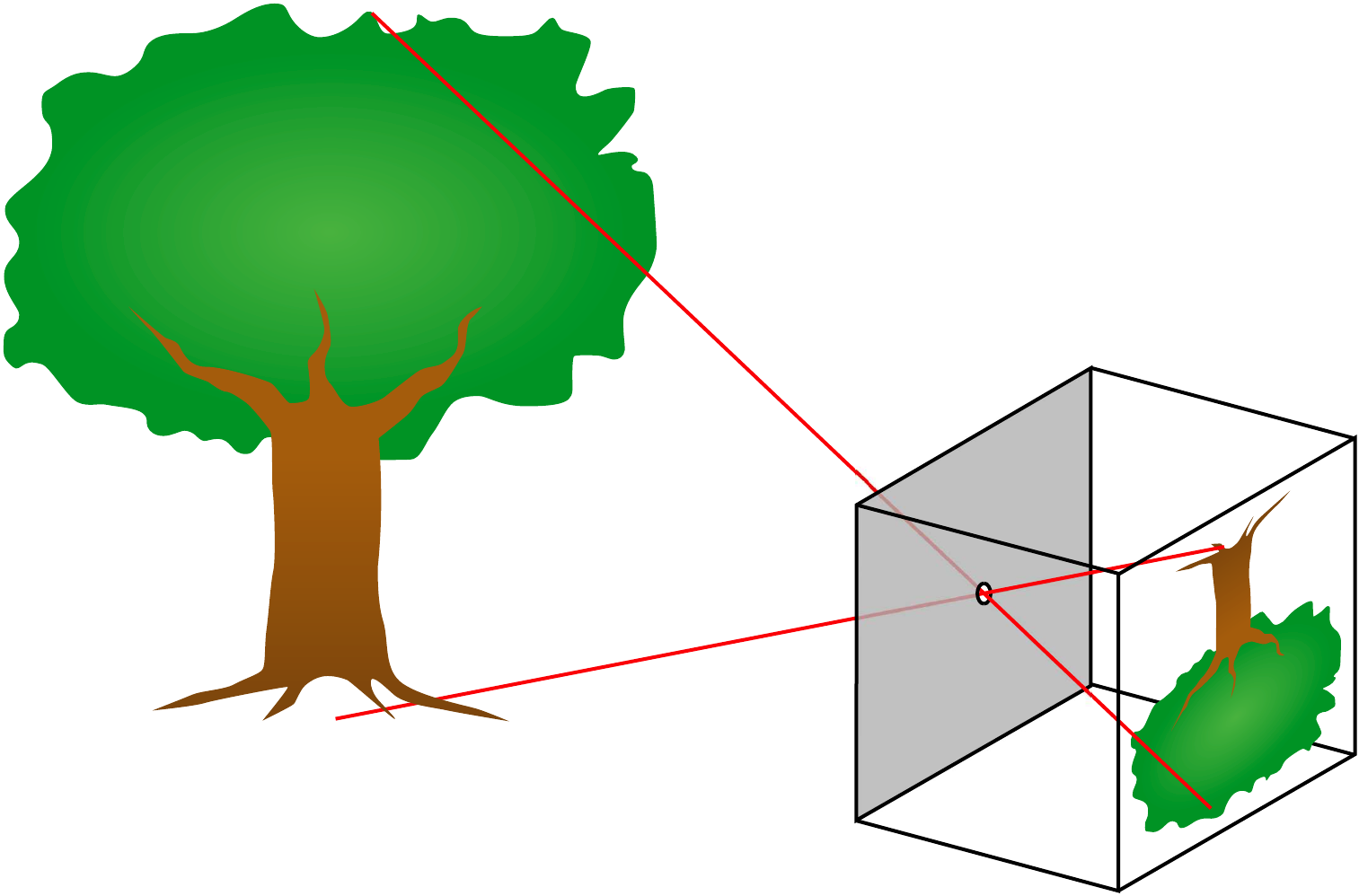}
\caption{A~pinhole camera, \url{http://en.wikipedia.org/wiki/File:Pinhole-camera.png}.}
\label{fig-camera}
\end{figure}

A central projection from $\mathbb{R}^3$ to $\mathbb{R}^2$ models a~pinhole camera pictured in
Fig.~\ref{fig-camera}.
It is described by a~linear fractional transformation 
\begin{gather}\label{proj}
x=\frac{p_{11} z_1+p_{12} z_2+p_{13} z_3+p_{14}}{p_{31} z_1+p_{32} z_2+p_{33} z_3+p_{34}},
\qquad
y=\frac{p_{21} z_1+p_{22} z_2+p_{23} z_3+p_{24}}{p_{31} z_1+p_{32} z_2+p_{33} z_3+p_{34}},
\end{gather}
where $(z_1,z_2,z_3)$ denote coordinates in $\mathbb{R}^3$,
$(x,y)$ denote coordinates in $\mathbb{R}^2$ and $p_{ij}$, $i=1,\dots,3$, $j=1,\dots,4$, are
real parameters of the projection, such that the left $3\times3$ submatrix of $3\times4$ matrix
$P=(p_{ij})$ has a~non-zero determinant.
Parameters represent the freedom to choose the center of the projection, the position of the image
plane and (in general, non-orthogonal) coordinate system on the image plane\footnote{It is clear
from~\eqref{proj} that multiplication of $P$ by a~non-zero constant does not change the projection
map.
Therefore, we can identify $P$ with a~point of the projective space~$\mathbb{PR}^{11}$, rather than
a~point in $\mathbb{R}^{12}$.
However, since we do not know which of the parameters are non-zero, in computations we have to keep
all 12 parameters.}. In the case when the distance between a~camera and an object is significantly
greater than the object depth, a~parallel projection provides a~good camera model.
A parallel projection has 8 parameters and can be described by a~$3\times4$ matrix of rank 3, whose
last row is $(0,0,0,1)$.
We review various camera models and related geometry in Section~\ref{cameras} (see
also~\cite{faugeras01, hartley04}).
In most general terms, the \emph{object-image correspondence problem}, or \emph{the projection
problem}, as we will call it from now on, can be formulated as follows:
\begin{problem}
\label{proj-problem}
Given a~subset ${\mathcal Z}$ of $\mathbb{R}^3$ and a~subset ${\mathcal X}$ of $\mathbb{R}^2$,
determine whether there exists a~projection $P\colon\mathbb{R}^3-\to\mathbb{R}^2$, such that
${\mathcal X}=P({\mathcal Z})$?\footnote{We borrow the notation $\Phi\colon V-\to W$ for
a~rational map from $V$ to $W$ from the algebraic geometry literature (see for
instance,~\cite{CLO96}) in order to emphasize that the map is defined almost everywhere on~$V$.
We use the same letter~$P$ to denote the $3\times4$ matrix $P=(p_{ij})$ and the map
$P\colon\mathbb{R}^3-\to\mathbb{R}^2$ defined by~\eqref{proj}.
Hence $P({\mathcal Z})=\{(x,y)\in\mathbb{R}^2\,|\,\exists\, {\mathbf z}\in{\mathcal Z}\text{ such
that~\eqref{proj} holds}\}$.}
\end{problem}

A straightforward approach to this problem consists of setting up a~system of conditions on the
projection parameters and then checking whether or not this system has a~solution.
In the case when ${\mathcal Z}$ and ${\mathcal X}$ are finite lists of points, a~solution based
on the straightforward approach can be found in~\cite{hartley04}.
For curves and surfaces under central projections, this approach is taken in~\cite{feldmar95}.
However, internal parameters of the camera are considered to be known in that paper and, therefore,
there are only 6 camera parameters in that study versus~12 considered here.
The method presented in~\cite{feldmar95} also uses an additional assumption that a~planar curve
${\mathcal X}\subset\mathbb{R}^2$ has at least two points, whose tangent lines coincide.
An alternative approach to the problem in the case when ${\mathcal Z}$ and ${\mathcal X}$ are
finite lists of points under parallel projections was presented
in~\cite{stiller07,stiller06}.
In these articles, the authors establish polynomial relationships that have to be satisfied by
coordinates of the points in the sets ${\mathcal Z}$ and ${\mathcal X}$ in order for a~projection
to exists.

Our approach to the projection problem for {\em curves} is somewhere in between the direct approach
and the implicit approach.
We exploit the relationship between the projection problem and equivalence problem under
group-actions to find the conditions that need to be satisfied by the object, the image and
the center of the projection\footnote{In the case of parallel projection, when the center is at
infinity, the conditions are on the direction of the projection.}. In comparison with the
straightforward approach, our solution leads to a~{\em significant reduction} of the number of
parameters that have to be eliminated in order to solve Problem~\ref{proj-problem} for curves.

All of the theoretical results of this paper are valid for arbitrary irreducible algebraic curves
(rational and non-rational), but the algorithms are presented for \emph{rational algebraic curves},
i.e.\
${\mathcal Z}=\overline{\{\Gamma(s)\,|\,s\text{ in the domain of }\Gamma\}}$ and ${\mathcal
X}=\overline{\{\gamma(t)\,|\, t\text{ in the domain of }\gamma\}}$ for rational maps
$\Gamma\colon\mathbb{R}-\to\mathbb{R}^3$ and $\gamma\colon\mathbb{R}-\to\mathbb{R}^2$.
A {\em bar} above a~set denotes the {\em Zariski closure} of the set\footnote{Recall that a~set
$\mathcal{W}\subset\mathbb{R}^n$ is Zariski closed if it equals to the zero set of a~system of
polynomials in $n$ variables.
The complement of a~Zariski closed set is called Zariski open.
A~Zariski open set is dense in $\mathbb{R}^n$.
A~Zariski closure~$\overline W$ of a~set~$W$ is the smallest (with respect to inclusions) Zariski
closed set containing~$W$.}.

Throughout the paper, we assume that ${\mathcal Z}$ is not a~straight line (and, therefore, its
image under any projection is a~one-dimensional constructible set).
Since, in general, $P({\mathcal Z})$ is not Zariski closed we must relax the projection condition
to $\overline{P({\mathcal Z})}={\mathcal X}$.
Under those conditions, Problem~\ref{proj-problem}, for central projections, can be reformulated as
the following real quantifier elimination problem:
\begin{reformulation}[straightforward approach]
\label{proj-problem-v2}
Given two rational maps $\Gamma$ and $\gamma$, determine the truth of the statement:
\begin{gather*}
\exists \, P\in\mathcal U\subset\mathbb{R}^{3\times4}
\quad
\forall\, s\text{ in the domain of }\Gamma
\quad
\exists\, t\in\mathbb{R}
\quad
P(\Gamma(s))=\gamma(t),
\end{gather*}
where $\mathcal U$ is the open subset of the set of $3\times4$ matrices defined by the condition
that the left $3\times3$ minor is nonzero\footnote{Note that,
in Reformulation~\ref{proj-problem-v2}, we decide whether ${P({\mathcal
Z})}\subset{\mathcal X}$, which appears to be weaker than $\overline{P({\mathcal Z})}={\mathcal X}$.
However, they are actually equivalent.
Since we assumed that ${\mathcal Z}$ is not a~line, the set $\overline{P({\mathcal Z})}$ is
one-dimensional.
Since ${\mathcal X}$ is rational algebraic curve, it is irreducible.
Hence ${P({\mathcal Z})}\subset{\mathcal X}
\
\Longleftrightarrow
\
\overline{P({\mathcal Z})}={\mathcal X}$.}.
\end{reformulation}

Real quantifier elimination problems are algorithmically solvable~\cite{tarski:51}.
A survey of subsequent developments in this area can be found, for instance, in~\cite{hong93}
and~\cite{caviness-johnson}.
Due to their high computational complexity (at least exponential) on the number of quantified
parameters, it is crucial to reduce the number of quantified parameters.
The main contribution of this paper is to provide another formulation of the problem which involves
significantly smaller number of quantified parameters.

We first begin by reducing the projection problem to the problem of deciding whether the given
planar curve ${\mathcal X}$ is equivalent to a~curve in a~certain family of {\em planar} curves
under an action of the projective group in the case of central projections, and under the action of
the affine group in the case of parallel projections.
The family of curves depends on 3 parameters in the case of central projections, and on 2
parameters in the case of parallel projections.

Then we solve these group-equivalence problems by an adaptation of differential signature
construction developed in~\cite{calabi98} for solving {\em local} equivalence problems for smooth
curves.
We give an algebraic formulation of the signature construction and show that it leads to a~solution
of {\em global} equivalence problems for algebraic curves.
For this purpose, we introduce a~notion of a~{\em classifying set of rational differential
invariants} and obtain such sets of invariants for the actions of the projective and affine
groups on the plane.
Following this method for the case of central projections, when ${\mathcal Z}$ and ${\mathcal X}$
are rational algebraic curves, we define two rational signature maps ${\mathcal S}|_{\mathcal
X}\colon\mathbb{R}-\to\mathbb{R}^2$ and ${\mathcal S}|_{\mathcal
Z}\colon\mathbb{R}^4-\to\mathbb{R}^2$.
Construction of these signature maps requires only differentiation and arithmetic operations and
is computationally trivial.
Then Problem~\ref{proj-problem-v2} becomes equivalent to \begin{reformulation}[signature approach]
\label{proj-problem-sig}
Given two rational maps ${S}|_{\mathcal X}$ and ${S}|_{\mathcal Z}$, determine the truth of the
statement:
\begin{gather*}
\exists\, c\in\mathcal{U}\subset\mathbb{R}^{3}
\quad
\forall\, s\text{ in the domain of }{S}_{\mathcal Z}(c,s)
\quad
\exists\, t\in\mathbb{R}
\quad
{S}_{\mathcal Z}(c,s)={S}_{\mathcal X}(t),
\end{gather*}
where $\mathcal U$ is a~certain Zariski open subset of $\mathbb{R}^3$.
\end{reformulation}
Note that Reformulations~\ref{proj-problem-v2} and~\ref{proj-problem-sig} have similar structure, but the former requires elimination of
14~parameters $(p_{11},\ldots,p_{34},s,t)$, while the latter requires elimination of only 5~parameters $(c_1,c_2,c_3,s,t)$.
The case of parallel projection is treated in the similar manner and leads to the reduction of the
number of real parameters that need to be eliminated from~10 to~4.

Although the relation between projections and group actions is known, our literature search did not
yield algorithms that exploit this relationship to solve the projection problem for curves in the
generic setting of cameras with unknown internal and external parameters.
The goal of the paper is to introduce such algorithms.
The significant reduction of the number of parameters in the quantifier elimination problem is
the main advantage of such algorithms.

\looseness=-1
A preliminary report on this project appeared in the conference proceedings~\cite{bk12}.
The current paper is significantly more comprehensive and rigorous, and also includes proofs
omitted in~\cite{bk12}.
Although the development of efficient implementation lies outside of the scope of this paper,
we made a~preliminary implementation of an algorithm based on signature construction presented here
and an algorithm based on the straightforward approach \emph{over complex numbers}.
The {\sc Maple} code and the experiments are posted on the internet~\cite{maple-code}.
The existence of a~projection over complex numbers provides necessary but not sufficient
condition for existence of a~real projection.

\looseness=-1
The paper is structured as follows.
In Section~\ref{cameras}, we review the basic facts about projections and cameras.
In Section~\ref{criteria}, we prove projection criteria that reduce the central and the parallel
projection problems to a~certain modification of the projective and the affine
group-equivalence problems for planar curves.
This criteria are straightforward consequences of known camera decompositions~\cite{hartley04}.
In Section~\ref{group-equiv}, we define the notion of a~\emph{classifying set of rational
differential invariants} and present a~solution of the \emph{global} group-equivalence problem
for planar algebraic curves based on these invariants.
This is an algebraic reformulation of a~solution of \emph{local} group-equivalence problem for
smooth curves~\cite{calabi98}.
In Section~\ref{algorithms}, combining the ideas from the previous two sections, we present and
prove an algorithm for solving the projection problem for rational algebraic curves and give
examples.
In Section~\ref{sect-ext}, we discuss possible adaptations of this algorithm to solve projection
problem for \emph{non-rational} algebraic curves and for finite lists of points.
We discuss the subtle difference between the discrete (with finitely many points) and the
continuous projection problems, showing that the solution for the discrete problem does not provide
an immediate solution to the projection problem for the curves represented by samples of points.
This leads us into the discussion of challenges that arise in application of our algorithms to
real-life images, given by discrete pixels, and of ideas for overcoming these challenges.
In Appendix~\ref{appendix}, we give explicit formulae for affine and projective
{classifying sets of rational invariants}.

\section{Projections and cameras}
\label{cameras}

We embed $\mathbb{R}^n$ into projective space $\mathbb{PR}^n$ and use homogeneous coordinates on
$\mathbb{PR}^n$ to express the map~\eqref{proj} by matrix multiplication.
\begin{notation} \emph{Square brackets} around matrices (and, in particular, vectors) will be used
to denote an equivalence class with respect to multiplication of a~matrix by a~nonzero scalar.
Multiplication of equivalence classes of matrices $A$ and $B$ of appropriate sizes is
well-defined by $[A] [B]:=[A\,B]$.
\end{notation} With this notation, a~point $(x,y)\in\mathbb{R}^2$ corresponds to a~point
$[x,y,1]=[\lambda x,\lambda y,\lambda]\in\mathbb{PR}^2$ for all $\lambda\neq0$, and a~point
$(z_1,z_2,z_3)\in\mathbb{R}^3$ corresponds to $[z_1,z_2,z_3,1]\in\mathbb{PR}^3$.
We will refer to the points in $\mathbb{PR}^n$ whose last homogeneous coordinate is zero as
\emph{points at infinity}.
In homogeneous coordinates projection~\eqref{proj} is a~map
$[P]\colon\mathbb{PR}^3\to\mathbb{PR}^2$ given by
\begin{gather*}
[x,y,1]^{\tr}=[P] [z_1,z_2,z_3,1]^{\tr},
\end{gather*}
where $P$ is $3\times4$ matrix of rank~3 and superscript $\tr$ denotes transposition.
Matrix $P$ has a~$1$-di\-mensional kernel.
Therefore, there exists a~point $[z^0_1,z^0_2,z^0_3,z^0_4]\in\mathbb{PR}^3$ whose image under the
projection is undefined (recall that $[0,0,0]$ is not a~point in $\mathbb{PR}^2$).
Geometrically, this point is the center of the projection.

In computer science literature (e.g.~\cite{hartley04}), a~camera is called {\it finite} if its
center is not at infinity.
A finite camera is modeled by a~matrix $P$, whose left $3\times3$ submatrix is non-singular.
Geometrically, finite cameras correspond to central projections from $\mathbb{R}^3$ to a~plane.
On the contrary, an \emph{infinite} camera has its center at an infinite point of~$\mathbb{PR}^3$.
An infinite camera is modeled by a~matrix $P$ whose left $3\times3$ submatrix is singular.
An infinite camera is called \emph{affine} if the preimage of the line at infinity in
$\mathbb{PR}^2$ is the plane at infinity in $\mathbb{PR}^3$.
An affine camera is modeled by a~matrix $P$ whose last row is $(0,0,0,1)$.
In this case map~\eqref{proj} becomes
\begin{gather*}
x=p_{11} z_1+p_{12} z_2+p_{13} z_3+p_{14},
\qquad
y=p_{21} z_1+p_{22} z_2+p_{23} z_3+p_{24}.
\end{gather*}
Geometrically, affine cameras correspond to {\em parallel projections} from
$\mathbb{R}^3$ to a~plane\footnote{Parallel projections are also called {\em generalized weak
perspective projections}~\cite{stiller07, stiller06}.}. Eight degrees of freedom reflect a~choice
of the direction of a~projection, a~position of the image plane and a~choice of linear system of
coordinates on the image plane.
In fact, by allowing the freedom to choose a~non-orthogonal coordinate system on the image plane,
we may always assume that we project on one of the coordinate planes.
\begin{definition}
A set of equivalence classes $[P]$, where $P$ is a~$3\times4$ matrix whose left $3\times3$
submatrix is non-singular, is called \emph{the set of central projections} and is denoted~$\fp$.

A set of equivalence classes $[P]$, where $P$ has rank 3 and its last row is $(0,0,0,\lambda)$,
$\lambda\neq0$, is called \emph{the set of parallel projections} and is denoted $\ap$.
\end{definition}

Equation~\eqref{proj} determines a~{\em central projection} when $[P]\in\fp$ and
it determines a~parallel projection when $[P]\in\ap$.
Sets $\fp$ and $\ap$ are disjoint.
Projections that are not included in these two classes correspond to infinite, non-affine
cameras.
These are not frequently used in computer vision and are not considered in this paper.

\section{Reduction to the group-equivalence problem}\label{criteria}

\begin{definition}
We say that a~curve ${\mathcal Z}\subset\mathbb{R}^3$ {\em projects to} ${\mathcal
X}\subset\mathbb{R}^2$ if there exists a~$3\times4$ matrix $P$ of rank 3 such that ${\mathcal
X}=\overline{P({\mathcal Z})}$, where
\begin{gather*}
P({\mathcal Z})=
\{(x,y)\in\mathbb{R}^2\,|\,\exists\, {\mathbf z} \in {\mathcal Z}\text{ such that \eqref{proj} holds}\}.
\end{gather*}
\end{definition}
Recall that for every algebraic curve ${\mathcal X}\subset\mathbb{R}^n$ there
exists a~unique projective algebraic curve $[{\mathcal X}]\subset\mathbb{PR}^n$ such that
$[{\mathcal X}]$ is the smallest projective variety containing ${\mathcal X}$ (see~\cite{fulton}).
It is not difficult to check that ${\mathcal X}=\overline{P({\mathcal Z})}$ is equivalent to
$[{\mathcal X}]=\overline{[P][{\mathcal Z}]}$, where $[P][{\mathcal Z}]=\left\{[P][{\mathbf
z}]\, |\,[{\mathbf z}]\in[{\mathcal Z}],\,[P][{\mathbf z}]\neq[0]\right\}$.

\begin{definition}
The \emph{projective group}\footnote{We will occasionally include a~field in
the group-notation, e.g.~${\mathcal{PGL}}(n,\mathbb{C})$ or ${\mathcal{PGL}}(n,\mathbb{R})$.
If the field is not indicated we assume that the group is defined over $\mathbb{R}$.}
${\mathcal{PGL}}(n+1)$ is a~quotient of the general linear group ${\mathcal{GL}}(n+1)$, consisting
of $(n+1)\times(n+1)$ non-singular matrices, by a~1-dimensional abelian subgroup $\lambda I$, where
$\lambda\neq0\in\mathbb{R}$ and $I$ is the identity matrix.
Elements of ${\mathcal{PGL}}(n+1)$ are equivalence classes $[B]=[\lambda B]$, where $\lambda\neq0$
and $B\in{\mathcal{GL}}(n+1)$.

The \emph{affine} group ${\mathcal A}(n)$ is a~subgroup of ${\mathcal{PGL}}(n+1)$ whose
elements $[B]$ have a~representative $B\in{\mathcal{GL}}(n+1)$ with the last row equal to
$(0,\dots,0,1)$.

The \emph{equi-affine} group ${\mathcal{SA}}(n)$ is a~subgroup of ${\mathcal A}(n)$ whose
elements $[B]$ have a~representative $B\in{\mathcal{GL}}(n+1)$ with determinant $1$ and the last
row equal to $(0,\dots,0,1)$.
\end{definition} In homogeneous coordinates, the standard action of the projective group
${\mathcal{PGL}}(n+1)$ on $\mathbb{PR}^n$ is defined by multiplication
\begin{gather}
\label{homog-act}
[z_1,\dots,z_n,z_0]^{\tr}\to[B]\,[z_1,\dots,z_n,z_0]^{\tr}.
\end{gather}
The action~\eqref{homog-act} induces linear-fractional action of ${\mathcal{PGL}}(n+1)$ on
$\mathbb{R}^n$.\footnote{Linear-fractional action of ${\mathcal{PGL}}(n+1)$ on $\mathbb{R}^n$ is an
example of a~rational action of an algebraic group on an algebraic variety.
General definition of a~rational action can be found in~\cite[Definition~2.1]{hk:focm}.} The
restriction of~\eqref{homog-act} to~${\mathcal A}(n)$ induces an action on~$\mathbb{R}^n$
consisting of compositions of linear transformations and translations.

\begin{definition}
We say that two curves ${\mathcal X}_1\subset\mathbb{R}^n$ and ${\mathcal X}_2\subset\mathbb{R}^n$
are ${\mathcal{PGL}}(n+1)$-equivalent if there exists $[A]\in{\mathcal{PGL}}(n+1)$, such that
$[{\mathcal X}_2]=\left\{[A][{\mathbf p}]\, | \,[{\mathbf p}]\in[{\mathcal X}_1]\right\}$.
We then write ${\mathcal X}_2=A\cdot{\mathcal X}_1$ or $[{\mathcal X}_2]=[A][{\mathcal X}_1]$.
If $[A]\in G$, where $G$ is a~subgroup of ${\mathcal{PGL}}(n+1)$, we say that ${\mathcal X}_1$ and
${\mathcal X}_2$ are $G$-equivalent and write ${\mathcal X}_1\underset G{\cong}{\mathcal X}_2$.
\end{definition}

Before stating the projection criteria, we make the following simple, but
important observations.
\begin{proposition}\label{proj-classes}\qquad\null

\begin{enumerate}\itemsep=0pt \item [$(i)$] If ${\mathcal Z}\subset\mathbb{R}^3$ projects to ${\mathcal
X}\subset\mathbb{R}^2$ by a~parallel projection, then any curve that is ${\mathcal
A}(3)$-equivalent to ${\mathcal Z}$ projects to any curve that is ${\mathcal A}(2)$-equivalent to
${\mathcal X}$ by a~parallel projection.
In other words, parallel projections are defined on affine equivalence classes of curves.
\item[$(ii)$] If ${\mathcal Z}\subset\mathbb{R}^3$ projects to ${\mathcal X}\subset\mathbb{R}^2$ by
a~central projection then any curve in $\mathbb{R}^3$ that is ${\mathcal A}(3)$-equivalent
to ${\mathcal Z}$ projects to any curve on $\mathbb{R}^2$ that is ${\mathcal{PGL}}(3)$-equivalent
to ${\mathcal X}$ by a~central projection.
\end{enumerate}
\end{proposition}
\begin{proof}
$(i)$ Assume that there exists a~parallel projection $[P]\in\ap$ such that $[{\mathcal
X}]=\overline{[P][{\mathcal Z}]}$.
Then for all $(A,B)\in\affp$ we have
\begin{gather*}
[A]\,[{\mathcal X}]=[A]\,\overline{[P]\,[B^{-1}]\,[B]\,[{\mathcal Z}]}=
\overline{[A]\,[P]\,[B^{-1}]\,[B]\,[{\mathcal Z}]}.
\end{gather*}
Since $[A]\,[P]\,[B^{-1}]\in\ap$, curve $B\cdot{\mathcal Z}$ projects to $A\cdot{\mathcal X}$.
$(ii)$ is proved similarly.
\end{proof}

\begin{remark}
\label{rem-pr-gr}
It is known that if ${\mathcal X}_1$ and ${\mathcal X}_2$ are images of a~curve ${\mathcal Z}$
under two central projections {\em with the same center}, then ${\mathcal X}_1$ and ${\mathcal
X}_2$ are ${\mathcal{PGL}}(3)$-equivalent, but if the centers of the projections are not the same
this is no longer true (see Example~\ref{ex-cp-cubics}).
Similarly, images of ${\mathcal Z}$ under various parallel projections may not be ${\mathcal
A}(2)$-equivalent.
\end{remark}

\begin{theorem}[central projection criterion]
\label{main-finite-camera}
A curve ${\mathcal Z}\subset\mathbb{R}^3$ projects to a~curve ${\mathcal X}\subset\mathbb{R}^2$ by
a~central projection if and only if there exist $c_1,c_2,c_3\in\mathbb{R}$ such that
${\mathcal X}$ is ${\mathcal{PGL}}(3)$-equivalent to a~planar curve
\begin{gather}
\label{poset}
\tilde{\mathcal Z}_{c_1,c_2,c_3}=\overline{\left\{\left(\frac{z_1+c_1}{z_3+c_3},
 \frac{z_2+c_2}{z_3+c_3}\right) \Big| (z_1,z_2,z_3)\in{\mathcal Z}\right\}}.
\end{gather}
\end{theorem}

\begin{proof}
($\Rightarrow$) Assume there exists a~central projection $[P]$ such that ${\mathcal
X}=\overline{P({\mathcal Z})}$.
Then $P$ is a~$3\times4$ matrix, whose left $3\times3$ submatrix is non-singular.
Therefore there exist $c_1,c_2,c_3\in\mathbb{R}$ such that
$p_{*4}=c_1 p_{*1}+c_2 p_{*2}+c_3 p_{*3}$, where $p_{*j}$ denotes the $j$-th column of the
matrix $P$.
We observe that
\begin{gather}
\label{eq-fp-dec}
[A][\stpf][B]=[P],
\end{gather}
where $A$ is the left $3\times3$ submatrix of $P$,
\begin{gather}
\label{eq-fPB}
\stpf:=\left(
\begin{matrix}1&0&0&0
\\
0&1&0&0
\\
0&0&1&0
\end{matrix}
\right) \qquad \text{and} \qquad B:=\left(
\begin{matrix}
1&0&0&c_1
\\
0&1&0&c_2
\\
0&0&1&c_3
\\
0&0&0&1
\end{matrix}
\right).
\end{gather}
Note that $[A]$ belongs to ${\mathcal{PGL}}(3)$.
Since
\begin{gather*}
[\stpf][B][z_1,z_2,z_3,\,1]^{\tr}=[z_1+c_1,z_2+c_2,z_3+c_3]^{\tr},
\end{gather*}
then $[{\mathcal X}]=[A][\tilde{\mathcal Z}_{c_1,c_2,c_3}]$, where $\tilde{\mathcal
Z}_{c_1,c_2,c_3}$ is defined by~\eqref{poset}.

($\Leftarrow$) To prove the converse direction we assume that there exists
$[A]\in{\mathcal{PGL}}(3)$ and $c_1,c_2,c_3\in\mathbb{R}$ such that $[{\mathcal
X}]=[A][\tilde{\mathcal Z}_{c_1,c_2,c_3}]$, where $\tilde{\mathcal Z}_{c_1,c_2,c_3}$ is defined
by~\eqref{poset}.
A direct computation shows that ${\mathcal Z}$ is projected to ${\mathcal X}$ by the central
projection $[P]=[A]\,[\stpf]\,[B]$, where $B$ and $[\stpf]$ are given by~\eqref{eq-fPB}.
\end{proof}

We note that the map
\begin{gather}
\label{eq-canc}
x=\frac{z_1+c_1}{z_3+c_3},
\qquad
y=\frac{z_2+c_2}{z_3+c_3}
\end{gather}
is a~projection centered $(-c_1,-c_2,-c_3)$ to the plane $z_3=1$ with coordinates on the image
plane induced from $\mathbb{R}^3$, namely, $x=z_1$ and $y=z_2$.
We call~\eqref{eq-canc} \emph{the canonical projection centered at $(-c_1,-c_2,-c_3)$}.
It follows from decomposition~\eqref{eq-fp-dec} that any central projection is a~composition of
a~translation in $\mathbb{R}^3$ (corresponding to translation of the camera center to the origin),
the canonical projection $\stpf$ centered at the origin, and a~projective transformation on the
image plane.

\begin{remark}[$\fp$ is a~homogeneous space] It is easy to check that the map
\[
{\Psi}: \ \left(\pta\right)\times\fp\to\fp
\]
 defined by
\begin{gather}
\label{product-action1}
{\Psi}\big(([A],[B]), [P])=[A] [P] \big[B^{-1}\big]
\end{gather}
for $[P]\in\fp$ and $([A],[B])\in\pta$ is an action of the product group~$\pta$ on the set of
central projections $\fp$.
Decomposition~\eqref{eq-fp-dec} shows that this action is transitive.
The stabilizer of the canonical projection $\stpf$ centered at the origin is a~$9$-dimensional group
\begin{gather*}
H_{\mathcal C}^0=\left\{\left([A], \left[
\begin{matrix}
A&\bf{0}^{\tr}
\\
\bf{0}&1
\end{matrix}
\right]\right)\right\},\qquad \text{where} \quad A\in{\mathcal{GL}}(3).
\end{gather*}
The set of central projections $\fp$ is, therefore, diffeomorphic to the homogeneous space
\[
\pta/H_{\mathcal C}^0.
\]
\end{remark}

\begin{theorem}[parallel projection criterion]
\label{main-affine-camera}
A curve ${\mathcal Z}\subset\mathbb{R}^3$ projects to a~curve ${\mathcal X}\subset\mathbb{R}^2$ by
a~parallel projection if and only if there exist $c_1,c_2\in\mathbb{R}$ and an ordered triplet
$(i,j,k)\in\left\{(1,2,3), (1,3,2), (2,3,1)\right\}$ such that ${\mathcal X}$ is ${\mathcal
A}(2)$-equivalent to
\begin{gather}
\label{delta-set}
\tilde{\mathcal Z}^{i,j,k}_{c_1,c_2}=\overline{\left\{\left(z_i+c_1 z_k,
z_j+c_2 z_k\right) \,\big|\, (z_1,z_2,z_3)\in{\mathcal Z}\right\}}.
\end{gather}
\end{theorem}

\begin{proof}
($\Rightarrow$) Assume there exists a~parallel projection $[P]$ such that ${\mathcal
X}=\overline{P({\mathcal Z})}$.
Then $[P]$ can be represented by a~matrix
\begin{gather*}
P=\left(
\begin{matrix}
p_{11}&p_{12}&p_{13}&p_{14}
\\
p_{21}&p_{22}&p_{23}&p_{24}
\\
0&0&0&1
\end{matrix}
\right)
\end{gather*}
of rank 3.
Therefore there exist $1\leq i<j\leq3$ such that the rank of the submatrix $\left(
\begin{matrix}p_{1i}&p_{1j}
\\
p_{2i}&p_{2j}
\end{matrix}
\right)$ is 2.
Then for $1\leq k\leq3$, such that $k\neq i$ and $k\neq j$, there exist $c_1, c_2\in\mathbb{R}$,
such that
$\left(
\begin{matrix}p_{1k}
\\
p_{2k}
\end{matrix}
\right)=c_1 \left(
\begin{matrix}p_{1i}
\\
p_{2i}
\end{matrix}
\right)+c_2 \left(
\begin{matrix}p_{1j}
\\
p_{2j}
\end{matrix}
\right)$.
We define
$A:=\left(
\begin{matrix}p_{1i}&p_{1j}&p_{14}
\\
p_{2i}&p_{2j}&p_{24}
\\
0&0&1
\end{matrix}
\right)$
and define $B$ to be the matrix whose columns are vectors $b_{*i}:=(1,0,0,0)^{\tr}$,
$b_{*j}:=(0,1,0,0)^{\tr}$, $b_{*k}:=(c_1,c_2,1,0)^{\tr}$, $b_{*4}=(0,0,0,1)^{\tr}$.
We observe that
\begin{gather}
\label{eq-pp-decomp}
[P]=[A][\stpa][B],\qquad \text{where} \quad \stpa:=\left(
\begin{matrix}
1&0&0&0
\\
0&1&0&0
\\
0&0&0&1
\end{matrix}
\right).
\end{gather}
Since $[\stpa][B][{\mathcal Z}]=[\tilde{\mathcal Z}^{i,j,k}_{c_1,c_2}]$, then $[{\mathcal
X}]=[A][\tilde{\mathcal Z}^{i,j,k}_{c_1,c_2}]$.
Observe that $[A]\in{\mathcal A}(2)$ and the direct statement is proved.

($\Leftarrow$) To prove the converse direction we assume that there exist $[A]\in{\mathcal A}(2)$,
two real numbers~$c_1$ and~$c_2$, and a~triplet of indices such that
$(i, j, k)\in\left\{(1,2,3), (1,3,2), (2,3,1)\right\}$, such that $[{\mathcal
X}]=[A][\tilde{\mathcal Z}^{i,j,k}_{c_1,c_2}]$, where a~planar curve $\tilde{\mathcal
Z}^{i,j,k}_{c_1,c_2}$ is given by~\eqref{delta-set}.
Let $B$ be a~matrix defined in the first part of the proof.
A direct computation shows that ${\mathcal Z}$ is projected to ${\mathcal X}$ by the parallel
projection $[P]=[A][\stpa][B]$.
\end{proof}

\begin{remark}[$\ap$ is a~homogeneous space] The map ${\Psi}:\left(\affp\right)\times\ap\to\ap$
defined by~\eqref{product-action1} for $[P]\in\ap$ and $([A],[B])\in\affp$ is an action of the
product group $\affp$ on the set of parallel projections~$\ap$.
Decomposition~\eqref{eq-pp-decomp} shows that this action is transitive.
The stabilizer of the orthogonal projection $\stpa$ is a~$10$-dimensional group
\begin{gather*}
H_{\mathcal P}^0= \left\{ \left( \left[
\begin{matrix}m_{11}&m_{12}&a_1
\\
m_{21}&m_{22}&a_2
\\
0&0&1
\end{matrix}
\right]
\!, \left.\left[
\begin{matrix}m_{11}&m_{12}&0&a_1
\\
m_{21}&m_{22}&0&a_2
\\
m_{31}&m_{32}&m_{33}&a_3
\\
0&0&0&1
\end{matrix}
 \right]\right)\right|
 m_{33} (m_{11}m_{22}-m_{12}m_{21})\neq0\right\} .
\end{gather*}
The set of central projections $\ap$ is, therefore, diffeomorphic to the homogeneous space
\[
\affp/H_{\mathcal P}^0.
\]
\end{remark}

The families of curves $\tilde{\mathcal Z}^{i,j,k}_{c_1,c_2}$ given by~\eqref{delta-set} have
a~large overlap.
The following corollary eli\-mi\-nates this redundancy and, therefore, is useful for practical
computations.

\begin{corollary}[reduced parallel projection criterion]
\label{reduced-aff-camera}
A curve ${\mathcal Z}\subset\mathbb{R}^3$ projects to ${\mathcal X}\subset\mathbb{R}^2$ by
a~parallel projection if and only if there exist $a_1,a_2,b\in\mathbb{R}$ such that the curve
${\mathcal X}$ is ${\mathcal A}(2)$-equivalent to one of the following planar curves:
\begin{gather}
\nonumber
\tilde{\mathcal Z}_{a_1,a_2}
=\overline{\left\{\left(z_1+a_1 z_3, z_2+a_2 z_3\right)\,\big|\,(z_1,z_2,z_3)\in{\mathcal Z}\right\}},
\\
\tilde{\mathcal Z}_b
=\overline{\left\{\left(z_1+b z_2, z_3\right)\,\big|\,(z_1,z_2,z_3)\in{\mathcal Z}\right\}},
\qquad
\tilde{\mathcal Z}
=\overline{\left\{\left(z_2, z_3\right)\,\big|\,(z_1,z_2,z_3)\in{\mathcal Z}\right\}}.\label{tosetcf}
\end{gather}
\end{corollary}

\begin{proof}
We first prove that for any permutation $(i,j,k)$ of numbers $(1,2,3)$ such that $i<j$, and for
any $c_1,c_2\in\mathbb{R}$ the set $\tilde{\mathcal
Z}^{i,j,k}_{c_1,c_2}=\left\{\left(z_i+c_1 z_k, z_j+c_2 z_k\right)\,\big|\,(z_1,z_2,z_3)\in{\mathcal Z}\right\}$
is ${\mathcal A}(2)$-equivalent to one of the sets listed in~\eqref{tosetcf}.

Obviously, $\tilde{\mathcal Z}^{1,2,3}_{c_1,c_2}=\tilde{\mathcal Z}_{a_1,a_2}$ with $a_1=c_1$ and
$a_2=c_2$.

For $\tilde{\mathcal Z}^{1,3,2}_{c_1,c_2}$, if $c_2\neq0$ then $\left(
\begin{matrix}1&-\frac{c_1}{c_2}
\\
0&\frac{1}{c_2}
\end{matrix}
\right)\left(
\begin{matrix}z_1+{c_1}z_2
\\
z_3+c_2{z_2}
\end{matrix}
\right)=\left(
\begin{matrix}
z_1-\frac{c_1}{c_2}z_3
\\
z_2+\frac1{c_2}{z_3}
\end{matrix}
\right)$ and so $\tilde{\mathcal Z}^{1,3,2}_{c_1,c_2}$ is ${\mathcal A}(2)$-equivalent to
$\tilde{\mathcal Z}_{a_1,a_2}$ with $a_1=-\frac{c_1}{c_2}$ and $a_2=\frac1{c_2}$.
Otherwise, if $c_2=0$, the $\tilde{\mathcal Z}^{1,3,2}_{c_1,c_2}=\tilde{\mathcal Z}_b$ with $b=c_1$.

Similarly for $\tilde{\mathcal Z}^{2,3,1}_{c_1,c_2}$, if $c_2\neq0$ then $\tilde{\mathcal
Z}^{2,3,1}_{c_1,c_2}$ is ${\mathcal A}(2)$-equivalent to $\tilde{\mathcal Z}_{a_1,a_2}$ with
$a_1=\frac{1}{c_2}$ and $a_2=-\frac{c_1}{c_2}$.
Otherwise, if $c_2=0$, then $\tilde{\mathcal Z}^{2,3,1}_{c_1,c_2}=(z_2+c_1z_1,
z_3)$.
If $c_1\neq0$ then $\tilde{\mathcal Z}^{2,3,1}_{c_1,c_2}$ is ${\mathcal A}(2)$-equivalent to
$\tilde{\mathcal Z}_b$ with $b=\frac1{c_1}$, otherwise $c_1=0$ and $\tilde{\mathcal
Z}^{2,3,1}_{c_1,c_2}=\tilde{\mathcal Z}$.

We can reverse the argument and show that any curve given by~\eqref{tosetcf} is ${\mathcal
A}(2)$-equivalent to a~curve from family~\eqref{delta-set}.
Then the reduced criteria follows from Theorem~\ref{main-affine-camera}.
\end{proof}

We note that the map
\begin{gather}
\label{eq-canp}
x={z_1+a_1 z_3},
\qquad
y=z_2+a_2 z_3
\end{gather}
is a~parallel projection onto the $z_1,z_2$-coordinate plane in the direction of the vector
$(-a_1,-a_2,1)$ with coordinates on the image plane induced from $\mathbb{R}^3$, namely, $x=z_1$
and $y=z_2$.
We call~\eqref{eq-canp} \emph{the canonical projection in the direction $(-a_1,-a_2,1)$}.
The map $x={z_1+b\,z_2}$,
$y=z_3$ is a~projection onto the $z_1,z_3$-coordinate plane in the direction of the vector
$(-b,1,0)$ with coordinates on the image plane induced from $\mathbb{R}^3$, namely, $x=z_1$ and
$y=z_3$, and finally the map $x=z_2$, $y=z_3$ is the orthogonal projection onto the $z_2,z_3$-plane.

\section{Solving the group-equivalence problem}\label{group-equiv}

Theorems~\ref{main-finite-camera} and~\ref{main-affine-camera} reduce the projection problem
to the problem of establishing group-action equivalence between a~given curve and a~curve from
a~certain family.
In this section, we give a~solution of the group-equivalence problem for planar algebraic curves.
In Section~\ref{ssec-inv}, we consider a~rational action of an \emph{arbitrary} algebraic group on~$\mathbb{R}^2$ and define a~notion of a~\emph{classifying set of rational differential
invariants}.
In Section~\ref{ssec-sig}, we define a~notion of \emph{exceptional curves} with respect to
a~classifying set of invariants and define \emph{signatures} of non-exceptional curves.
We then prove that signatures characterize the equivalence classes of non-exceptional curves.
In Section~\ref{ssec-ap-inv}, we produce explicit formulae for classifying sets of {rational}
differential invariants for \emph{affine and projective groups}.
In Section~\ref{ssec-rat-sig}, we specialize our signature construction to \emph{rational algebraic
curves} and provide examples of solving group-equivalence problem for such curves.

We note that differential invariants have long been used for solving the group-equivalence
problem for smooth curves.
Classical differential invariants were obtained with the moving frame method~\cite{C37}, which
most often produces \emph{non-rational} invariants.
Signatures based on classical differential invariants were introduced in~\cite{calabi98}.
For smooth curves, the equality of signatures of two curves implies that there are segments of two
curves that are group-equivalent (in other words, these curves are \emph{locally} equivalent), but
the entire curves may be non-equivalent.
This is well illustrated in~\cite{musso09}.
In a~recent work~\cite{ho12}, a~significantly more involved notion of the extended signature was
introduced to solve global equivalence problem for smooth curves.

The rigidity of irreducible algebraic curves allows us to use simpler signatures to establish
\emph{global} equivalence.
Rationality of the invariants as well as explicit characterization of exceptional curves allows us
to solve equivalence problem using standard computational algebra algorithms.

\subsection{Definition of a~classifying set of rational differential invariants}
\label{ssec-inv}

A rational action of an algebraic group $G$ on $\mathbb{R}^2$ can be prolonged to an action on the
$n$-th jet space $J^n=\mathbb{R}^{n+2}$ with coordinates $(x,y,y^{(1)},\dots,y^{(n)})$ as
follows\footnote{Here $y=y^{(0)}$ and $J^0=\mathbb{R}^2$.}. For a~fixed $g\in G$, let $(\bar
x,\bar y)=g\cdot(x,y)$.
Then $\bar x$, $\bar y$ are rational functions of $(x,y)$ and
\begin{gather}
g\cdot\big(x,y,y^{(1)},\dots,y^{(n)}\big)=
\big(\bar x,\bar y,\bar y^{(1)},\dots,\bar y^{(n)}\big),\label{eq-pra}
\end{gather}
where
\begin{gather*}
\bar y^{(1)}=
\frac{\frac{d}{dx} \big[\bar y(x,y)\big]}{\frac{d}{dx}\,\big[\bar x(x,y)\big]}\!\qquad \text{and for}\!\quad k=
1,\dots,{n-1}
\!\qquad
\bar y^{(k+1)}=
\frac{\frac{d}{dx} \left[\bar y^{(k)}(x,y,y^{(1)},\dots,y^{(k)})\right]}{\frac{d}{dx}\left[\bar x(x,y)\right]}.
\end{gather*}
Here $\frac{d}{dx}$ is the total derivative, applied under assumption that $y$ is
function of~$x$.\footnote{We note the duality of our view of variables $y^{(k)}$.
On one hand, they are viewed as independent coordinate functions on $J^n$.
On the other hand, operator $\frac{d}{dx}$ is applied under assumption that $y$ is a~function of
$x$ and, therefore, $y^{(k)}$~is also viewed as the $k$-th derivative of~$y$ with respect to~$x$.
(The same duality of view appears in calculus of variations.)} We note that a~natural projection
$\pi^n_k\colon J^n\to J^k$, $k<n$ is equivariant with respect to action~\eqref{eq-pra}.
For general theory of rational actions see~\cite{vinberg89} and for general definitions and
properties of the jet bundle and prolongations of actions see~\cite{olver:yellow}.

\begin{definition} A function on~$J^n$ is called a~{\em differential function}.
The {\em order of a~differential function} is the maximum value of $k$ such that the function
explicitly depends on the variable~$y^{(k)}$.
A differential function which is invariant under action~\eqref{eq-pra} is called a~{\em differential invariant}.
\end{definition}

\begin{remark} Due to equivariant property of the projection $\pi^n_k\colon J^n\to J^k$, $k<n$,
a~differential invariant of order $k$ on $J^k$ can be viewed as a~differential invariant on
$J^n$ for all $n\geq k$.
\end{remark}

\begin{definition}
Let $G$ act on $\mathbb{R}^N$.
We say that a~set $\mathcal I$ of invariant rational functions on $\mathbb{R}^N$ {\em separates}
orbits on a~subset $W\subset\mathbb{R}^N$ if $W$ is contained in the domain of definition of each
$I\in\mathcal I$ and $\forall\, w_1,w_2\in W$
\begin{gather*}
I(w_1)=I(w_2) \quad \forall\, I\in\mathcal I\ \Longleftrightarrow\ \exists\, g\in G\text{ such that }w_1=
g\cdot w_2.
\end{gather*}
\end{definition}

 \begin{definition}[classifying set of rational differential invariants]
\label{def-dsep}
Let $r$-dimensional algebraic group $G$ act on $\mathbb{R}^2$.
Let $K$ and $T$ be \emph{rational differential invariants} of orders ${r-1}$ and $r$,
respectively.
The set ${\mathcal I}=\{K,T\}$ is called {\em classifying } if $K$ separates orbits on a~Zariski
open subset $W^{r-1}\subset J^{r-1}$ and ${\mathcal I}=\{K,T\}$ separates orbits on a~Zariski open
subset $W^r\subset J^{r}$.
\end{definition}

\subsection{Jets of curves and signatures}
\label{ssec-sig}
In this section, we assume that ${\mathcal X}\subset\mathbb{R}^2$ is an irreducible algebraic
curve, different from a~vertical line.
Let $F(x,y)$ be an irreducible polynomial, whose zero set equals to~${\mathcal X}$.
Then the derivatives of~$y$ with respect to~$x$ are rational functions on~${\mathcal X}$, whose
explicit formulae are obtained by implicit differentiation
\begin{gather*}
y^{(1)}_{{\mathcal X}}=-\frac{F_x}{F_y},
\qquad
y^{(2)}_{{\mathcal X}}=\frac{-F_{xx}\,F_y^2+2\,F_{xy}\,F_x\,F_y-F_{yy}\,F_x^2}{F_y^3},
\qquad
\dots.
\end{gather*}

\begin{definition}
\label{def-cjet}
The $n$-th jet of a~curve ${\mathcal X}\subset\mathbb{R}^2$ is a~rational map $j^n_{{\mathcal
X}}\colon{\mathcal X}-\to J^n$, where for ${\mathbf p}\in{\mathcal X}$
\begin{gather*}
j^n_{{\mathcal X}}({\mathbf p})=
\big(x({\mathbf p}), y({\mathbf p}),
y^{(1)}_{{\mathcal X}}({\mathbf p}),\dots,y^{(n)}_{{\mathcal X}}({\mathbf p})\big).
\end{gather*}
\end{definition}
From the definition of the prolonged action~\eqref{eq-pra}, it follows that for
all $g\in G$ and ${\mathbf p}\in{\mathcal X}$ the following equality holds, whenever both sides
are defined.
\begin{gather}
\label{eq-jg}
j^n_{g\cdot{\mathcal X}}(g\cdot{\mathbf p})=g\cdot\big[j^n_{\mathcal X}({\mathbf p})\big].
\end{gather}

\begin{definition}
\label{def-restr}
A \emph{restriction of a~rational differential function} $\Phi\colon J^n-\to\mathbb{R}$ to
a~curve ${\mathcal X}$ is a~composition of $\Phi$ with the $n$-th jet of curve, i.e.\
$\Phi|_{{\mathcal X}}=\Phi\circ j^n_{{\mathcal X}}$.
If defined, such composition produces a~rational function ${\mathcal X}-\to\mathbb{R}$.
\end{definition}

\begin{definition}
\label{def-excep}
Let $\mathcal I=\{K,T\}$ be a~{classifying set of rational differential invariants} for
$G$-action (see Definition~\ref{def-dsep}).
Then a~point ${\mathbf p}\in{\mathcal X}$ is called \emph{${\mathcal I}$-regular} if:
\begin{enumerate} \itemsep=0pt
\item[(1)] ${\mathbf p}$ is a~non-singular point of ${\mathcal X}$;
\item[(2)]
$j^{r-1}_{{\mathcal X}}({\mathbf p})\in W^{r-1}$ and $j^{r}_{{\mathcal X}}({\mathbf p})\in W^{r}$;
\item[(3)] $\left.\frac{\partial K}{\partial y^{(r-1)}}\right|_{j^{r-1}_{{\mathcal X}}({\mathbf
p})}\neq0$ and $\left.\frac{\partial T}{\partial y^{(r)}}\right|_{j^{r}_{{\mathcal X}}({\mathbf
p})}\neq0$.
\end{enumerate}
An algebraic curve ${\mathcal X}\subset\mathbb{R}^2$ is called {\em
non-exceptional} with respect to~${\mathcal I}$ if all but a~finite number of its points are
${\mathcal I}$-regular.
\end{definition}

\begin{lemma}
\label{lemma-sigv}
Let $\mathcal I=\{K,T\}$ be a~{classifying set of rational differential invariants} $($see
Definition~{\rm \ref{def-dsep})}.
Let ${\mathcal X}\subset\mathbb{R}^2$ be a~non $\mathcal I$-exceptional curve defined by an
irreducible implicit equation $F(x,y)=0$.
Then \begin{enumerate}\itemsep=0pt
 \item[$(1)$] $K|_{\mathcal X}$ and $T|_{\mathcal X}$ are rational functions on
${\mathcal X}$ and therefore there exist polynomials $k_1,k_2\in\mathbb{R}[x,y]$ with no
non-constant common factors modulo $F$, and polynomials ${t_1},{t_2}\in\mathbb{R}[x,y]$ with no
non-constant common factors modulo $F$, such that
\begin{gather}
\label{eq-KP}
K|_{\mathcal X}(x,y)=\frac{k_1(x,y)}{k_2(x,y)}\qquad \text{and} \qquad T(x,y)|_{\mathcal X}=
\frac{t_1(x,y)}{t_2(x,y)}.
\end{gather}
\item[$(2)$] The Zariski closure $\overline{{\mathcal S}_{{\mathcal X}}}$ of the image of
the rational map ${S}|_{\mathcal X}\colon{\mathcal X}-\to\mathbb{R}^2$, defined by
${S}|_{\mathcal X}({\mathbf p})=(K|_{{\mathcal X}}({\mathbf p}), T|_{{\mathcal X}}({\mathbf p}))$
for ${\mathbf p}\in{\mathcal X}$, is the variety of the elimination ideal $\hat
X=X\cap\mathbb{R}[\varkappa,\tau]$, where
\begin{gather}
\label{eq-X}
X:=
\langle F,k_2 \varkappa-k_1, t_2 \tau-t_1, k_2 t_2 \sigma-1\rangle\subset\mathbb{R}[\varkappa,\tau,x,y,\sigma].
\end{gather}
\item[$(3)$] $\dim\overline{{\mathcal S}_{{\mathcal X}}}=0$ if and only if $K_{\mathcal X}$ and
$T_{\mathcal X}$ are constant functions on ${\mathcal X}$ and $\dim\overline{{\mathcal
S}_{{\mathcal X}}}=1$ otherwise.
In the latter case, $\overline{{\mathcal S}_{{\mathcal X}}}$ is an irreducible algebraic planar
curve, i.e.\ a~zero set of an irreducible polynomial ${\hat S}_{{\mathcal X}}(\varkappa,\tau)$.
\end{enumerate}
\end{lemma}
\begin{proof}
(1)~A~function $K|_{\mathcal X}\colon{\mathcal X}-\to\mathbb{R}$ is
a~composition of rational maps
\begin{gather*}
j^{r-1}_{{\mathcal X}}\colon \ {\mathcal X}-\to J^{r-1}\qquad \text{and} \qquad K\colon J^{r-1}-\to\mathbb{R}
\end{gather*}
(see Definitions~\ref{def-cjet} and~\ref{def-restr}).
Since ${\mathcal X}$ is non-exceptional this composition is defined for all but finite number
of points on ${\mathcal X}$ and therefore $K|_{\mathcal X}$ is a~rational function.
The same argument shows that~$T|_{\mathcal X}$ is a~rational function.
Since ${\mathcal X}$ is defined by $F(x,y)=0$, where $F$ is irreducible, there exist polynomials
$k_1,k_2\in\mathbb{R}[x,y]$ with no non-constant common factors modulo $F$, and polynomials
${t_1},{t_2}\in\mathbb{R}[x,y]$ with no non-constant common factors modulo $F$, such
that~\eqref{eq-KP} holds.

(2)~By definition,
\begin{gather*}
{\mathcal S}_{\mathcal X}=\left\{(\kappa,\tau)\, |\, \exists\, (x,y)
\
F(x,y)=0\,\wedge\, \varkappa=\frac{k_1(x,y)}{k_2(x,y)}\,\wedge\,\tau=\frac{t_1(x,y)}{t_2(x,y)}\right\}
\end{gather*}
and therefore is the projection of the variety defined by~$X$, given by~\eqref{eq-X}, to the
$\varkappa,\tau$-plane.
It is the standard theorem in the computational algebraic geometry (see, for
instance,~\cite[Chapter~5]{CLO96}) that the Zariski closure $\overline{{\mathcal S}_{\mathcal X}}$
of this projection is the variety of the elimination ideal $X\cap\mathbb{R}[\varkappa,\tau]$.

(3)~It is not difficult to prove, in general, that the Zariski closure of the image of
an irreducible variety under a~rational map is an irreducible variety.
The dimension of this closure is less or equal than the dimension of the original variety.
Thus $\overline{{\mathcal S}_{\mathcal X}}$ is an irreducible variety of dimension zero when the
signature map ${S}|_{\mathcal X}=(K|_{\mathcal X},T|_{\mathcal X})$ is a~constant map and of
dimension one otherwise.
In the latter case $\overline{{\mathcal S}_{\mathcal X}}$ is an irreducible algebraic planar curve
and, therefore, is a~zero set of a~single irreducible polynomial.
\end{proof}

\begin{definition}
\label{def-sig}
Let ${\mathcal I}=\{K,T\}$ be a~classifying set of rational differential invariants with respect
to $G$-action and ${\mathcal X}$ be non-exceptional with respect to ${\mathcal I}$.
\begin{enumerate}\itemsep=0pt
 \item[(1)] The rational map ${S}|_{\mathcal X}\colon{\mathcal X}-\to\mathbb{R}^2$
defined by ${S}|_{\mathcal X}({\mathbf p})=(K|_{{\mathcal X}}({\mathbf p}), T|_{{\mathcal
X}}({\mathbf p}))$ for ${\mathbf p}\in{\mathcal X}$ is called the \emph{signature map}.
\item[(2)] The image of ${S}|_{\mathcal X}$ is called the \emph{signature} of ${\mathcal X}$ and is
denoted by ${\mathcal S}_{{\mathcal X}}$.
\end{enumerate}
\end{definition}
\begin{theorem}[group-equivalence criterion]
\label{th-main-sig}
Assume that irreducible algebraic curves ${\mathcal X}_1$ and ${\mathcal X}_2$ are non-exceptional
with respect to a~classifying set of rational differential invariants ${\mathcal I}=(K,T)$ under
the $G$-action.
Then ${\mathcal X}_1$ and ${\mathcal X}_2$ are $G$-equivalent if and only if their signatures are
equal:
\begin{gather*}
{{\mathcal X}_1}\underset G{\cong}{{\mathcal X}_2}
\quad
\Longleftrightarrow
\quad
{\mathcal S}_{{\mathcal X}_1}={\mathcal S}_{{\mathcal X}_2}.
\end{gather*}
\end{theorem}
\begin{remark}
Assume that two curves ${\mathcal X}_1$ and ${\mathcal X}_2$ have non-constant signature maps, and
so the closures of their signatures are zero sets of polynomials ${\hat S}_{{\mathcal
X}_1}(\varkappa,\tau)$ and ${\hat S}_{{\mathcal X}_2}(\varkappa,\tau)$, respectively.
The equality of signatures ${\mathcal S}_{{\mathcal X}_1}={\mathcal S}_{{\mathcal X}_2}$, implies
that ${\hat S}_{{\mathcal X}_1}(\varkappa,\tau)$ is equal up to a~constant multiple to~${\hat
S}_{{\mathcal X}_2}(\varkappa,\tau)$.
The converse is true over $\mathbb{C}$, but not over $\mathbb{R}$, because the latter is not an
algebraically closed field (see Example~\ref{ex-psig} below and~\cite{CLO96} for general results
on implicitization).
\end{remark}

\begin{proof}[Proof of Theorem~\ref{th-main-sig}.]
Direction $\Longrightarrow$ follows immediately from the
definition of invariants.
Below we prove $\Longleftarrow$.
We notice that there are two cases.
Either $K|_{{\mathcal X}_1}$ and $K|_{{\mathcal X}_2}$ are constant maps on~${{\mathcal X}_1}$ and~${{\mathcal X}_2}$, respectively, and these maps take the same value.
Otherwise both~$K|_{{\mathcal X}_1}$ and~$K|_{{\mathcal X}_2}$ are non-constant rational maps on~${{\mathcal X}_1}$ and~${{\mathcal X}_2}$, respectively.

\emph{Case 1:} There exists $\varkappa_0\in\mathbb{R}$ such that $K|_{{\mathcal X}_1}({\mathbf
p}_1)=\varkappa_0$ and $K_{{\mathcal X}_2}({\mathbf p}_2)=\varkappa_0$ for all ${\mathbf
p}_1\in{\mathcal X}_1$ and for all ${\mathbf p}_2\in{\mathcal X}_2$.
Since ${\mathcal X}_1$ and ${\mathcal X}_2$ are non-exceptional, we may fix
$\mathcal{I}_G$-regular points ${\mathbf p}_1=(x_1,y_1)\in{\mathcal X}_1$ and ${\mathbf
p}_2=(x_2,y_2)\in{\mathcal X}_2$.
Then, due to separation property of the invariant~$K$, $\exists\, g\in G$ such that
$j^{r-1}_{{\mathcal X}_1}({\mathbf p}_1)=g\cdot[j^{r-1}_{{\mathcal X}_2}({\mathbf p}_2)]$.
We consider a~new algebraic curve ${{\mathcal X}_3}=g\cdot{\mathcal X}_2$.
Then due to~\eqref{eq-jg}, we have
\begin{gather}
\label{j13}
j^{r-1}_{{\mathcal X}_1}({\mathbf p}_1)=j^{r-1}_{{{\mathcal X}_3}}({\mathbf p}_1)=:\jp{r-1}.
\end{gather}

Since ${\mathbf p}_1$ is a~${\mathcal I}$-regular point of ${\mathcal X}_1$, it follows
from~\eqref{j13} that it is also a~${\mathcal I}$-regular point of~${{\mathcal X}_3}$ and, in
particular, is non-singular.
Let $F_1(x,y)=0$ and $F_3(x,y)=0$ be implicit equations of~${\mathcal X}_1$ and~${{\mathcal X}_3}$,
respectively.
We may assume that $\left.\pd{F_1}{y}\right|_{{\mathbf p}_1}\neq0$ and
$\left.\pd{F_3}{y}\right|_{{\mathbf p}_1}\neq0$ (otherwise, $\left.\pd{F_1}{x}\right|_{{\mathbf
p}_1}\neq0$ and $\left.\pd{F_3}{x}\right|_{{\mathbf p}_1}\neq0$ and we may use a~similar argument).
Then, there exist functions $f_1(x)$ and $f_3(x)$, analytic on an interval~$I\ni x_1$, such that
$F_1(x,f_1(x))=0$ and $F_3(x,f_3(x))=0$ for $x\in I_1$.

Functions $y=f_1(x)$ and $y=f_3(x)$ are local analytic solutions of differential equation
\begin{gather}
\label{eq-K=c}
K\big(x,y,y^{(1)},\dots,y^{(r-1)}\big)=\varkappa_0
\end{gather}
with the same initial condition $f_1^{(k)}(x_1)=f_3^{(k)}(x_1)$, $k=0,\dots,{r-1}$ prescribed
by~\eqref{j13}.
From the ${\mathcal I}$-regularity of ${\mathbf p}_1$, we have that $\left.\frac{\partial
K}{\partial y^{(r-1)}}\right|_{\jp{r-1}}\neq0$ and so~\eqref{eq-K=c} can be solved for $y^{(r-1)}$:
\begin{gather*}
y^{(r-1)}=H\big(x,y,y^{(1)},\dots,y^{(r-2)}\big),
\end{gather*}
where function $H$ is smooth in a~neighborhood $\jp{r-1}\in J^{r-1}$.
From the uniqueness theorem for the solutions of ODEs, it follows that $f_1(x)=f_3(x)$ on an
interval $I\ni x_1$.
Since ${\mathcal X}_1$ and ${{\mathcal X}_3}$ are irreducible algebraic curves it follows that
${\mathcal X}_1={{\mathcal X}_3}$.
Therefore, ${\mathcal X}_1=g\cdot{\mathcal X}_2$.

\emph{Case 2:} $K|_{{\mathcal X}_1}$ and $K|_{{\mathcal X}_2}$ are non-constant rational maps.
Then ${\mathcal S}_{{\mathcal X}_1}={\mathcal S}_{{\mathcal X}_2}$ is a~one-dimensional set that we
will denote ${\mathcal S}$.
Let ${\hat S}(\varkappa,\tau)=0$ be the implicit equation for ${\mathcal S}$ (see
Lemma~\ref{lemma-sigv}).
We know that $\pd{{\hat S}}{\tau}(\varkappa,\tau)\neq0$ for all but finite number of values
$(\varkappa,\tau)$, because, otherwise, $K|_{{\mathcal X}_1}$~and~$K|_{{\mathcal X}_2}$ are
constant maps.
Therefore, since the curves are non-exceptional, there exists $\mathcal{I}$-regular points
${\mathbf p}_1=(x_1,y_1)\in{\mathcal X}_1$ and ${\mathbf p}_2=(x_2,y_2)\in{\mathcal X}_2$ such that
\begin{gather}\label{eq-pf2}
K|_{{\mathcal X}_1}({\mathbf p}_1)=K|_{{\mathcal X}_2}({\mathbf p}_2)=:\varkappa_0,
\qquad
T|_{{\mathcal X}_1}({\mathbf p}_1)=T|_{{\mathcal X}_2}({\mathbf p}_2)=:\tau_0,
\qquad
\pd{{\hat S}}{\tau}(\varkappa_0,\tau_0)\neq0.
\end{gather}
Due to separation property of the set $\mathcal{I}_G=\{K,T\}$, $\exists\, g\in G$ such that
$j^{r}_{{\mathcal X}_1}({\mathbf p}_1)=g\cdot[j^{r}_{{\mathcal X}_2}({\mathbf p}_2)]$.
We consider a~new algebraic curve ${{\mathcal X}_3}=g\cdot{\mathcal X}_2$.
Then due to~\eqref{eq-jg}, we have
\begin{gather}
\label{j13-r}
j^{r}_{{\mathcal X}_1}({\mathbf p}_1)=j^{r}_{{{\mathcal X}_3}}({\mathbf p}_1)=:\jp r.
\end{gather}
From~\eqref{eq-pf2},~\eqref{j13-r} and $\mathcal{I}$-regularity of the point ${\mathbf
p}_1\in{\mathcal X}_1$ it follows that
\begin{gather}
\label{eq-kt0}
K\big(\jp r\big)=\varkappa_0,
\qquad
T\big(\jp r\big)=\tau_0\qquad \text{and} \qquad \left.\frac{\partial T}{\partial y^{(r)}}\right|_{\jp r}\neq0.
\end{gather}

Since ${\mathbf p}_1$ is a~${\mathcal I}$-regular point of ${\mathcal X}_1$, it follows
from~\eqref{j13-r} that it is also a~${\mathcal I}$-regular point of ${{\mathcal X}_3}$ and, in
particular, is non-singular.
Let $F_1(x,y)=0$ and $F_3(x,y)=0$ be implicit equations of~${\mathcal X}_1$ and~${{\mathcal X}_3}$,
respectively.
We may assume that $\left.\pd{F_1}{y}\right|_{{\mathbf p}_1}\neq0$ and
$\left.\pd{F_3}{y}\right|_{{\mathbf p}_1}\neq0$ (otherwise, $\left.\pd{F_1}{x}\right|_{{\mathbf
p}_1}\neq0$ and $\left.\pd{F_3}{x}\right|_{{\mathbf p}_1}\neq0$ and we may use a~similar argument).
Then, there exist functions~$f_1(x)$ and~$f_3(x)$, analytic on an interval $I\ni x_1$, such that
$F_1(x,f_1(x))=0$ and $F_3(x,f_3(x))=0$ for $x\in I_1$.
Then functions $y=f_1(x)$ and $y=f_3(x)$ are local analytic solutions of differential equation
\begin{gather}
\label{eq-KT}
{\hat S}\big(K\big(x,y,y^{(1)},\dots,y^{(r-1)}\big),T\big(x,y,y^{(1)},\dots,y^{(r)}\big)\big)=0
\end{gather}
 with the same initial condition $f_1^{(k)}(x_1)=f_3^{(k)}(x_1)$, $k=0,\dots,{r}$,
dictated by~\eqref{j13-r}.

Since $\pd{{\hat S}}{\tau}(\varkappa_0,\tau_0)\neq0$ and $\left.\frac{\partial T}{\partial
y^{(r)}}\right|_{\jp r}\neq0$ (see~\eqref{eq-pf2} and~\eqref{eq-kt0}), equation~\eqref{eq-KT} can
be solved for~$y^{(r)}$:
\begin{gather*}
y^{(r)}=H\big(x,y,y^{(1)},\dots,y^{(r-1)}\big),
\end{gather*}
where function $H$ is smooth in a~neighborhood $\jp r\in J^{r}$.
From the uniqueness theorem for the solutions of ODE it follows that $f_1(x)=f_3(x)$ on an interval
$I\ni x_1$.
Since ${\mathcal X}_1$ and ${\mathcal X}_3$ are irreducible algebraic curves it follows that
${\mathcal X}_1={\mathcal X}_3$.
Therefore, ${\mathcal X}_1=g\cdot{\mathcal X}_2$.
\end{proof}
From the proof of Theorem~\ref{th-main-sig} we may deduce the following:
\begin{corollary}\label{cor-cases}\qquad
\begin{enumerate}\itemsep=0pt
 \item[$1)$] Assume $K|_{{\mathcal X}_1}$ is a~constant function with value
$\varkappa_0\in\mathbb{R}$, i.e.
$K|_{{\mathcal X}_1}({\mathbf p})=\varkappa_0$ for all ${\mathbf p}\in{\mathcal X}_1$.
Then ${\mathcal S}|_{{\mathcal X}_2}={\mathcal S}|_{{\mathcal X}_1}$ if and only if $K|_{{\mathcal
X}_2}({\mathbf p})=\varkappa_0$ for all ${\mathbf p}\in{\mathcal X}_2$.
\item[$2)$] If $\dim({\mathcal
S}|_{{\mathcal X}_1}\cap{\mathcal S}|_{{\mathcal X}_2})=1$ then ${\mathcal S}|_{{\mathcal
X}_1}={\mathcal S}|_{{\mathcal X}_2}$.
\end{enumerate}
\end{corollary}

\subsection{Classifying sets of invariants for affine and projective actions}
\label{ssec-ap-inv}

In this section, we construct a~classifying set of rational differential invariants for
affine and projective actions.
We will build them from classical invariants from differential geometry~\cite{blaschke23,C37}.
We start with Euclidean curvature
\begin{gather}
\label{eq-kappa}
\kappa=\frac{y^{(2)}}{(1+[y^{(1)}]^2)^{3/2}},
\end{gather}
which is, up to a~sign\footnote{The sign of $\kappa$ changes when a~curve is reflected, rotated
by~$\pi$ radians or traced in the opposite direction.
A~rational function~$\kappa^2$ is invariant under the full Euclidean group.}, a~Euclidean
differential invariant of the lowest order.
Higher order Euclidean differential invariants are obtained by differentiating the curvature
with respect to the Euclidean arclength $ds={\sqrt{1+[y^{(1)}]^2}} dx$,
i.e.~$\kappa_s=\frac{d \kappa}{d s}=\frac1{\sqrt{1+[y^{(1)}]^2}}\frac{d \kappa}{dx}$,
$\kappa_{ss}=\frac{d \kappa_s}{d s},\dots$.

 Equi-affine and projective curvatures and infinitesimal arclengths are well known,
and can be expressed in terms of Euclidean invariants~\cite{faugeras94, kogan03}.
In particular, ${\mathcal{SA}}(2)$-curvature $\mu$ and infinitesimal
${\mathcal{SA}}(2)$-arclength $d\alpha$ are expressed in terms of their Euclidean counterparts as
follows
\begin{gather*}
\mu=\frac{3 \kappa (\kappa_{ss}+3 \kappa^3)-5 \kappa_s^2}{9 \kappa^{8/3}},
\qquad
d\alpha=\kappa^{1/3}ds.
\end{gather*}

\noindent By considering effects of scalings and reflections on ${\mathcal{SA}}(2)$-invariants,
we obtain two lowest order ${\mathcal A}(2)$-invariants
\begin{gather}
\label{aff-inv}
K_{\mathcal A}=\frac{(\mu_\alpha)^2}{\mu^3},
\qquad
T_{\mathcal A}=\frac{\mu_{\alpha\alpha}}{3 \mu^2}.
\end{gather}
They are of order 5 and 6, respectively, and are \emph{rational functions} in jet variables.

${\mathcal{PGL}}(3)$-curvature $\eta$ and infinitesimal arclength $d\rho$ are expressed in terms
of their ${\mathcal{SA}}$-counterparts
\begin{gather*}
\eta=\frac{6\mu_{\alpha\alpha\alpha}\mu_\alpha-7 \mu_{\alpha\alpha}^2-9\mu_\alpha^2 \mu}{6\mu_\alpha^{8/3}},
\qquad
d\rho=\mu_\alpha^{1/3}d\alpha.
\end{gather*}
The two lowest order \emph{rational} ${\mathcal{PGL}}(3)$-invariants are of differential order~7
and~8, respectively
\begin{gather}
\label{proj-inv}
K_{\mathcal P}=\eta^3,
\qquad
T_{\mathcal P}=\eta_\rho.
\end{gather}
Explicit formulae for invariants in terms of jet coordinate are given by~\eqref{eq-kta}
and~\eqref{eq-kp},~\eqref{eq-tp} in Appendix~\ref{appendix}.
\begin{theorem}
\label{th-sep-a-p}
According to Definition~{\rm \ref{def-dsep}}:
\begin{enumerate}\itemsep=0pt
 \item[$(1)$] The set
${\mathcal I}_{\mathcal A}=\{K_{\mathcal A},T_{\mathcal A}\}$ given by~\eqref{aff-inv} is
classifying for the ${\mathcal A}(2)$-action on $\mathbb{R}^2$.
\item[$(2)$] The set ${\mathcal I}_{\mathcal{P}}=\{K_{\mathcal P},T_{\mathcal P}\}$
given by~\eqref{proj-inv} is classifying for the ${\mathcal{PGL}}(3)$-action on $\mathbb{R}^2$.
\end{enumerate}
\end{theorem}

\begin{proof}
See the appendix for the explicit expressions
of $K_{\mathcal A},T_{\mathcal A}, K_{\mathcal P}$ and $T_{\mathcal P}$.
Observe that the denominators only involve
\begin{gather}
\label{eq-D1}
\Delta_1:=3 \jy4 \jy2-5 \big[\jy3\big]^2,
\\
\label{eq-D2}
\Delta_2:=9 \jy5 \big[\jy2\big]^2-45 \jy4 \jy3 \jy2+40 \big[\jy3\big]^3.
\end{gather}
Hence, as long as they are non-zero, the expressions are well-defined.

Let us prove~(1). Note that $\dim{\mathcal A}(2)=6$.
We will prove the separation property of ${\mathcal I}_{\mathcal A}$ on the~Zariski open subset
\[W^6=\big\{\mathbf{p}\in J^6\,\big|\,\jy2\neq0\text{ and }\Delta_1(\mathbf{p})\neq0\big\}\] of $J^6$,
and the separation property of $K_{\mathcal
A}$ on $W^5=\pi^6_5(W^6)\subset J^5$.
One can check directly that $W^6$ (and hence $W^5$) is invariant under affine transformation and, {moreover, ${\rm sign}(\Delta_1)={\rm sign}(K_{\mathcal A})$ is invariant under any affine transformation.}

Taking into account subtleties of adaptation of the moving frame method to algebraic context (see~\cite{hk:focm}), we find the affine transformation
\[ A=\left(
\begin{matrix}
a_{11} & a_{12} & a_{13} \\
a_{21} & a_{22} & a_{23} \\
0 & 0 & 1
\end{matrix}
\right) \in {\mathcal A}(2)
\] with
\begin{gather*}
a_{11} = {\rm sign}(\Delta_2)\sqrt{|\Delta_1|} \frac{3 \big(\jy2\big)^2-\jy1 \jy3}{9 \big(\jy2\big)^3}, \qquad
a_{12} ={\rm sign}(\Delta_2) \sqrt{|\Delta_1|} \frac{\jy3}{9 \big(\jy2\big)^3}, \\
a_{13}={\rm sign}(\Delta_2)\sqrt{|\Delta_1|}\frac{x\jy1\jy3-3x\big(\jy2\big)^2-y\jy3}{9\big(\jy2\big)^3}, \\
a_{21}=|\Delta_1|\frac{-\jy1}{9 \big(\jy2\big)^3},\qquad
a_{22}=|\Delta_1| \frac{1}{9 \big(\jy2\big)^3},\qquad
a_{23}=|\Delta_1|\frac{x\jy1-y}{9\big(\jy2\big)^3},
\end{gather*}
that brings $\mathbf{p} \in W^6$ to
\[
A_{\mathbf p}\,\cdot \mathbf{p} =
\big(0,0,0,1,0, {\rm sign}(\Delta_1(\mathbf{p}))\cdot3,
3 \sqrt{|K_{\mathcal A}(\mathbf{p})|}, 9T_{\mathcal A}(\mathbf{p})+45\big).\]
Therefore,
if ${K_{\mathcal A}(\mathbf{p}_1)}={K_{\mathcal A}(\mathbf{p}_2)}$ and
 ${T_{\mathcal A}(\mathbf{p}_1)}={T_{\mathcal A}(\mathbf{p}_2)}$, then
${\rm sign}(\Delta_1(\mathbf{p}_1))={\rm sign}(\Delta_1(\mathbf{p}_2))$
and thus $A_{\mathbf{p}_1}\cdot\mathbf{p}_1= A_{\mathbf{p}_2}\cdot\mathbf{p}_2$
and in turn $\mathbf{p}_2 = A_{\mathbf{p}_2}^{-1} A_{\mathbf{p}_1}\cdot\mathbf{p}_1$.
Hence ${\mathcal I}_{\mathcal A}$ is separating on $W^6$. Separation property of $K_{\mathcal A}$ on $W^5$ is seen by disregarding the 6-th order jet component in the above argument and noticing that $A_{\mathbf p}$ is computed on the 4-th order jet space.

Let us prove~(2). Note that $\dim{\mathcal{PGL}}(3)=8$.
We will prove the separation property of ${\mathcal I}_{\mathcal P}$ on the~Zariski open subset
\begin{gather*}
W^8=\big\{\mathbf{p}\in J^8\,\big|\,\jy2\neq0,\, \Delta_1(\mathbf{p})\neq 0\text{ and }\Delta_2(\mathbf{p})\neq0\big\}
\end{gather*}
of $J^8$ and the separation property of $K_{\mathcal P}$
on $W^7=\pi^8_7(W^8)\subset J^7$.
One can check directly that $W^8$ (and hence $W^7$) is invariant under projective transformation.
We can find the projective transformation
\begin{gather*}
A=\left(
\begin{matrix}g&h g& 0
\\
0& g^2 & 0
\\
h &i&1
\end{matrix}
\right)
\left(
\begin{matrix}e&f&0
\\
0&\frac1e&0
\\
0&0&1
\end{matrix}
\right)\left(
\begin{matrix}c&-s&a
\\
s&c&b
\\
0&0&1
\end{matrix}
\right)\in{\mathcal{PGL}}(3)
\end{gather*}
with
\begin{gather*}
c=\frac{1}{\sqrt{1+\big[y^{(1)}\big]^2}},
\qquad\!
s=-\frac{y^{(1)}}{\sqrt{1+\big[y^{(1)}\big]^2}},
\qquad\!
a=-\frac{y^{(1)} y+x}{\sqrt{1+\big[y^{(1)}\big]^2}},
\qquad\!
b=\frac{y^{(1)} x-y}{\sqrt{1+\big[y^{(1)}\big]^2}}, \\
e=\big[\jy2_1\big]^{1/3},
\qquad
f=\frac{\jy3_1}{3 \big[\jy2_1\big]^{5/3}},
\qquad
g=\big[\jy5_2\big]^{1/3},
\qquad
h=\frac{5 \big[\jy4_2\big]^2-\jy6_2}{3 \jy5_2},\\
i=\frac{\big[\jy6_2\big]^2-10 \jy6_2\big[\jy4_2\big]^2-3\big[\jy5_2\big]^2 \jy4_2+25\big[\jy4_2\big]^4}{18\big[\jy5_2\big]^2}
\end{gather*}
 that brings $\mathbf{p} \in W^8$ to
\[
A_{\mathbf p}\cdot \mathbf{p} =\left(0,0,0,1,0,0,1,0,\left(K_{\mathcal P}\big(\mathbf{p}\big)\right)^{1/3},\frac 1 3 T_{\mathcal P}\big(\mathbf{p}\big)+\frac{105}6 \right)
\]
Therefore, if
${K_{\mathcal P}(\mathbf{p}_1)}={K_{\mathcal P}(\mathbf{p}_2)}$ and
${T_{\mathcal P}(\mathbf{p}_1)}={T_{\mathcal P}(\mathbf{p}_2)}$, then
$A_{\mathbf{p}_1}\cdot\mathbf{p}_1= A_{\mathbf{p}_2}\cdot\mathbf{p}_2$
and in turn $\mathbf{p}_2 = A_{\mathbf{p}_2}^{-1}\,A_{\mathbf{p}_1}\cdot\mathbf{p}_1$.
Hence ${\mathcal I}_{\mathcal P}$ is separating on $W^8$. Separation property of $K_{\mathcal P}$ on $W^7$ is seen by disregarding the 8-th order jet component in the above argument and noticing that $A_{\mathbf p}$ is computed on the 6-th order jet space.
\end{proof}

Theorem~\ref{th-sep-a-p}, in combination with Theorem~\ref{th-main-sig}, leads to a~solution for
the projective and the affine equivalence problems for non-exceptional curves.
The following proposition describes exceptional curves.

\begin{proposition}
 ${\mathcal I}_{\mathcal A}$- and ${\mathcal I}_{\mathcal{PGL}}$-exceptional curves are
lines and conics\footnote{By a~conic we mean an irreducible real algebraic planar curve of degree~$2$, i.e.~a parabola, an ellipse, or a~hyperbola.}.
\end{proposition}

\begin{proof}
In the affine case, we note that $j^{5}_{{\mathcal X}}({\mathbf p})\in J^5\setminus W^{5}$ and
$j^{6}_{{\mathcal X}}({\mathbf p})\in J^6\setminus W^{6}$ if an only if $\kappa|_{\mathcal
X}({\mathbf p})=0$ or $\mu|_{\mathcal X}({\mathbf p})=0$.
If $\kappa|_{\mathcal X}({\mathbf p})=0$ for more than finite number of points on ${\mathcal X}$,
then ${\mathcal X}$ is a~line, if $\mu|_{\mathcal X}({\mathbf p})=0$ for more than finite number
of points on ${\mathcal X}$ then it is a~parabola (see Proposition~\ref{prop-lines-conics}).
From the explicit formulae~\eqref{eq-kta} we see that
\begin{gather*}
\left.\frac{\partial K_{\mathcal A}}{\partial\jy5}\right|_{\jp5}=
\frac{18 \Delta_2\big[\jy2\big]^2}{\Delta_1^3}
\qquad \text{and}
\qquad
\left.\frac{\partial T_{\mathcal A}}{\partial y^{(6)}}\right|_{\jp6}=\frac{9 \big[\jy2\big]^3}{\Delta_1^2}.
\end{gather*}
These rational functions are not identically zero on ${\mathcal X}$ if neither $\Delta_2|_{\mathcal
X}=0$ nor $\jy2|_{\mathcal X}=0$, or equivalently ${\mathcal X}$ is not a~line or a~conic.
Therefore, if ${\mathcal X}$ is not a~line or a~conic, it is $\{\mathcal
K_{\mathcal A}, \mathcal T_{\mathcal A}\}$-regular.
In the projective case, we note that $j^{7}_{{\mathcal X}}({\mathbf p})\in J^7\setminus W^{7}$ and
$j^{8}_{{\mathcal X}}({\mathbf p})\in J^8\setminus W^{8}$ if an only if $\kappa|_{\mathcal
X}({\mathbf p})=0$ or $\mu_{\alpha}|_{\mathcal X}({\mathbf p})=0$.
If $\mu_\alpha|_{\mathcal X}({\mathbf p})=0$ for more than finite number of points on ${\mathcal
X}$ then ${\mathcal X}$ is a~conic (see Proposition~\ref{prop-lines-conics}).
From the explicit formulae~\eqref{eq-kp} and~\eqref{eq-tp}, we see that
\begin{gather*}
\left.\frac{\partial K_{\mathcal P}}{\partial\jy7}\right|_{\jp7}\neq0
\qquad
\text{and}
\qquad
\left.\frac{\partial T_{\mathcal P}}{\partial y^{(8)}}\right|_{\jp8}\neq0
\end{gather*}
for all $\jp7\in W^7$ and all $\jp8\in W^8$.
Therefore, if an algebraic curve is not a~line or a~conic it is $\{\mathcal K_{\mathcal
P}, \mathcal T_{\mathcal P}\}$-regular.
\end{proof}

\begin{remark}
\label{rem-excep}
It is well known (and easy to prove) that the set of all lines constitutes a~single equivalence
class (an orbit) under both ${\mathcal A}(2,\mathbb{R})$-action and
${\mathcal{PGL}}(3,\mathbb{R})$-action.
Under the ${\mathcal A}(2,\mathbb{R})$-action the set of all conics splits into three orbits: the
set of all parabolas, the set of all hyperbolas and the set of all ellipses, although under the
${\mathcal A}(2,\mathbb{C})$-action the set of all hyperbolas and ellipses comprise a~single orbit.
All conics constitute a~single orbit under ${\mathcal{PGL}}(3,\mathbb{R})$-action,
see~\cite[Section~II.5]{bix98}.
Therefore an ${\mathcal I}_{\mathcal A}$-exceptional algebraic curve is {\em not} ${\mathcal
A}(2,\mathbb{C})$-equivalent to a~non ${\mathcal I}_{\mathcal A}$-exceptional algebraic curve and
an ${\mathcal I}_{\mathcal{PGL}}$-exceptional algebraic curve is {\em not}
${\mathcal{PGL}}(3,\mathbb{C})$-equivalent to a~non ${\mathcal I}_{\mathcal{PGL}}$-exceptional
algebraic curve.
\end{remark}

The projective and the affine equivalence problems for exceptional curves can be
easily solved using the above remark and the following proposition.

\begin{proposition}
\label{prop-lines-conics}
Let ${\mathcal X}$ be an irreducible planar algebraic curve.
Then
\begin{enumerate}\itemsep=0pt
\item[$(1)$] ${\mathcal X}$ is a~line $\Longleftrightarrow$
$\kappa|_{\mathcal X}=0$;
\item[$(2)$] ${\mathcal X}$ is a~parabola $\Longleftrightarrow$
$\mu|_{\mathcal X}=0$;
\item[$(3)$] ${\mathcal X}$ is a~conic $\Longleftrightarrow$
$\mu_\alpha|_{\mathcal X}=0$,
 \end{enumerate}
 where ``$=0$'' means that a~corresponding rational
function is zero at every point of ${\mathcal X}$.
The above statements are true for both real and complex algebraic curves.
If ${\mathcal X}$ is a~real algebraic curve, then it is a~hyperbola if and only if $\mu|_{\mathcal
X}$ is a~negative constant, while ${\mathcal X}$ is an ellipse if and only if $\mu|_{\mathcal X}$
is a~positive constant.

\end{proposition} The proof of part (1) of Proposition~\ref{prop-lines-conics} follows immediately
from~\eqref{eq-kappa}.
Proofs of the other statements can be found in~\cite[Section~7.3]{Gug63}.
The following corollary is obtained from Proposition~\ref{prop-lines-conics} using explicit
formulae for equi-affine invariants
\begin{gather*}
\mu=\frac19\frac{\Delta_1}{\big[\jy2\big]^{8/3}},
\qquad
\mu_\alpha=\frac1{27}\frac{\Delta_2}{\big[\jy2\big]^4},
\end{gather*}
where $\Delta_1$ and $\Delta_2$ are given by~\eqref{eq-D1} and~\eqref{eq-D2}.

\begin{corollary}
\label{cor-delta}
Let ${\mathcal X}$ be an irreducible planar algebraic curve.
Assume that ${\mathcal X}$ is not a~\emph{vertical} line.
Let $\Delta_1$ and $\Delta_2$ be given by~\eqref{eq-D1} and~\eqref{eq-D2}.
Then
\begin{enumerate}\itemsep=0pt
\item[$(1)$] The restrictions $\Delta_1|_{\mathcal X}$ and
$\Delta_2|_{\mathcal X}$ are rational functions on ${\mathcal X}$.
\item[$(2)$] $\Delta_1|_{\mathcal X}$ is a~zero function if and only if ${\mathcal X}$ is a~line or
a~parabola.
Otherwise, restrictions of ${\mathcal A}(2)$-invariants $K_{\mathcal A}|_{\mathcal X}$ and
$T_{\mathcal A}|_{\mathcal X}$ are rational functions of ${\mathcal X}$.\footnote{$K_{\mathcal
A}|_{\mathcal X}$ and $T_{\mathcal A}|_{\mathcal X}$ are defined and are both zero functions when
${\mathcal X}$ is either an ellipse or a~hyperbola.
We know, however, that ellipses and hyperbolas are not ${\mathcal A}(2,\mathbb{R})$-equivalent.
There is no contradiction with Theorem~\ref{th-main-sig}, because ellipses and hyperbolas are
${\mathcal I}_{\mathcal A}$-exceptional per Definition~\ref{def-excep}.}
\item[$(3)$]
$\Delta_2|_{\mathcal X}$ is a~zero function if and only if ${\mathcal X}$ is a~line or a~conic.
Otherwise $K_{\mathcal P}|_{\mathcal X}$ and $T_{\mathcal P}|_{\mathcal X}$ are rational functions
on ${\mathcal X}$.
\end{enumerate}
\end{corollary}

\subsection{Signatures of rational curves}
\label{ssec-rat-sig}
In this section, we adapt the signature constructions to \emph{rational} algebraic curves and give
examp\-les of solving the affine and projective equivalence problems using signatures.
We adapt Definition~\ref{def-restr} to rational curves as follows.
Let ${\mathcal X}$ is a~rational curve parameterized by $\gamma(t)=(x(t),y(t))$, such that $x(t)$
is not a~constant function\footnote{Equivalently, ${\mathcal X}$ is not a~vertical line.}. Make
a~recursive definition of the following rational functions of $t$:
\begin{gather}\label{eq-ykt}
y^{(1)}=\frac{\dot y}{\dot x}
\qquad
,\dots,
\qquad
y^{(k)}=\frac{\dot{y^{(k-1)}}}{\dot x},
\end{gather}
where $\dot{\phantom{a}}$ \emph{denotes the derivative with respect to the parameter.} Let $\Phi$
be a~rational differential function.
Then the restriction of~$\Phi|_\gamma$ is computed by substituting~\eqref{eq-ykt} into~$\Phi$.
If defined, $\Phi|_\gamma$ is a~rational function of~$t$.
Recalling Definition~\ref{def-sig} of signature and Corollary~\ref{cor-delta}, we conclude that:
\begin{proposition}
Let ${\mathcal X}$ be an irreducible planar algebraic curve parameterized by a~rational map~$\gamma(t)$.
Assume $\Delta_2|_\gamma\underset{\mathbb{R}(t)}\neq0$, where $\Delta_2$ is given by~\eqref{eq-D2}.
Then
\begin{enumerate}\itemsep=0pt \item[$(1)$] the ${\mathcal A}(2)$-signature of ${\mathcal X}$ is the image of
a~rational map ${S}_{\mathcal A}|_\gamma(t):=\left(K_{\mathcal A}|_\gamma,T_{\mathcal
A}|_\gamma\right)$ from $\mathbb{R}$ to~$\mathbb{R}^2$.
\item[$(2)$] the ${\mathcal{PGL}}(3)$-signature of ${\mathcal X}$ is the image of a~rational map
${S}_{\mathcal P}|_\gamma(t):=\left(K_{\mathcal P}|_\gamma,T_{\mathcal P}|_\gamma\right)$ from
$\mathbb{R}$ to~$\mathbb{R}^2$.
\end{enumerate}
\end{proposition}

The signatures can be computed either using inductive formulae
for invariants given by~\eqref{aff-inv} and~\eqref{proj-inv}, or explicit formulae given
by~\eqref{eq-kta} and~\eqref{eq-kp},~\eqref{eq-tp} in Appendix~\ref{appendix}.
A~solution for ${\mathcal A}(2)$- and ${\mathcal{PGL}}(3)$-equivalence problems for rational curves
follows from Theorem~\ref{th-main-sig} and is illustrated by the following examples.
Invariants and signatures were computed using Maple code~\cite{maple-code}.

\begin{example}[${\mathcal{PGL}}(3)$-equivalence problems]
\label{ex-psig}
Consider three rational cubics pictured on Fig.~\ref{fig-cubics}.
\begin{figure}[t]
\centering \includegraphics[scale=.276]{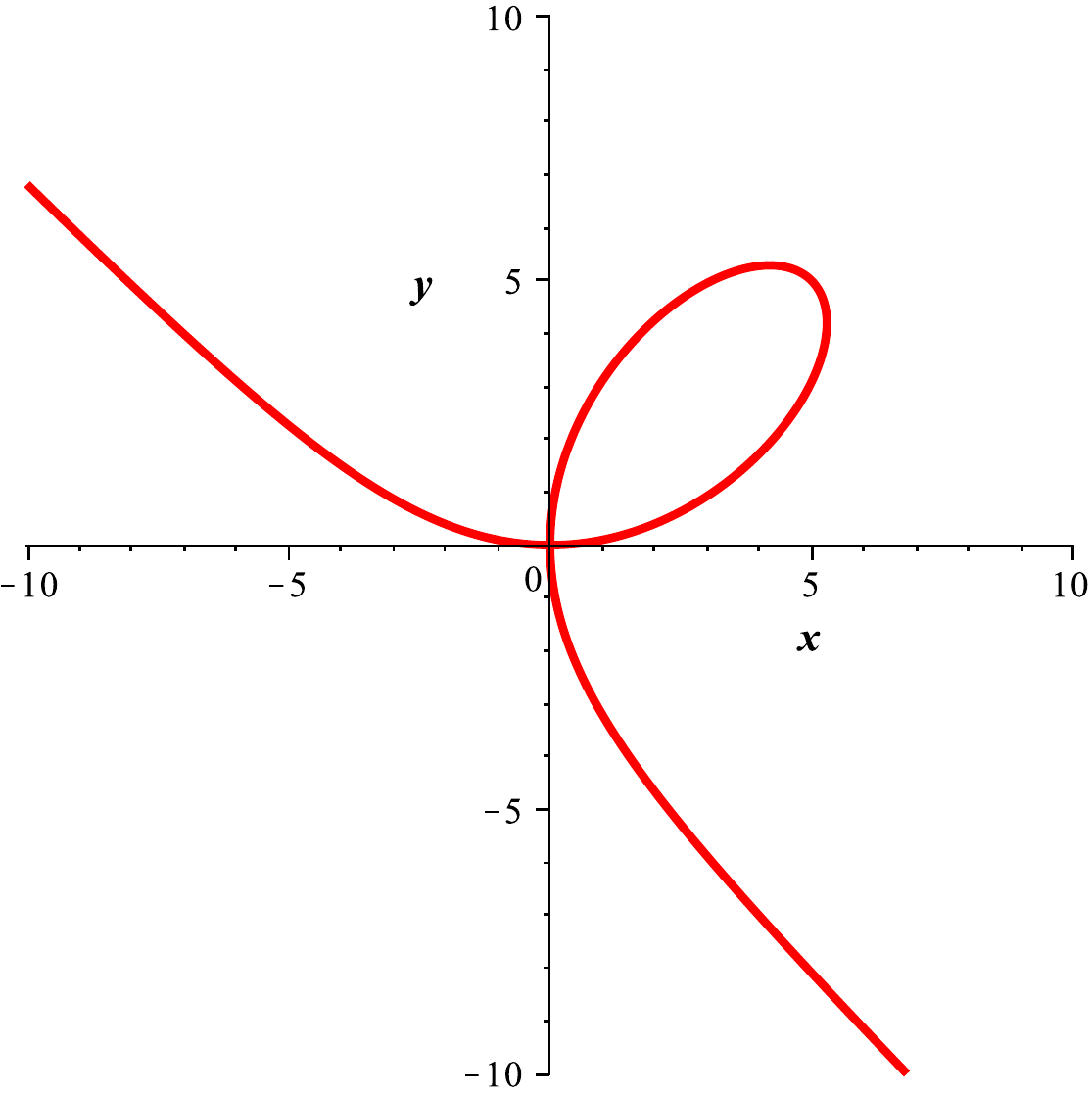} \includegraphics[scale=.26]{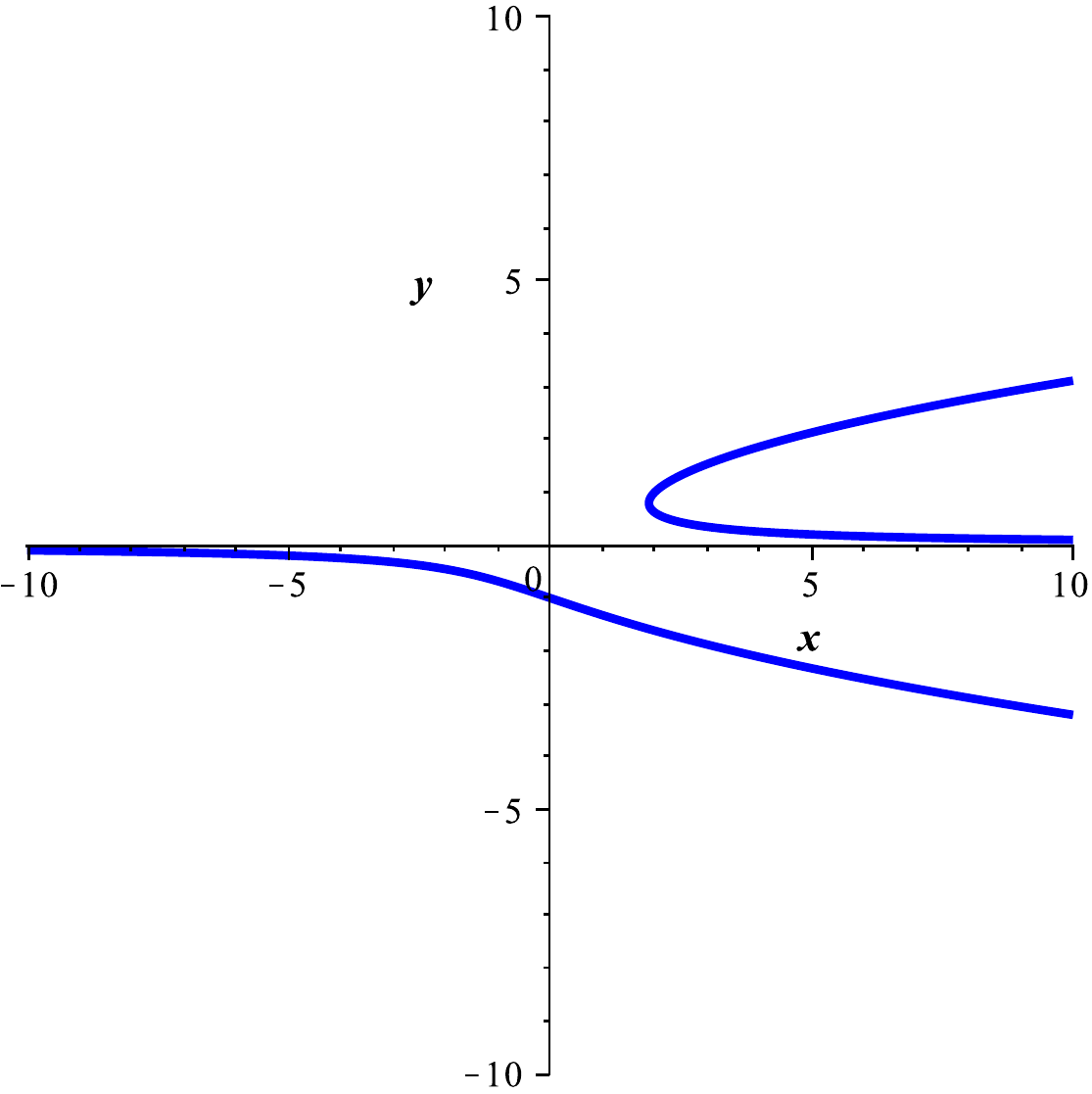}
\includegraphics[scale=.26]{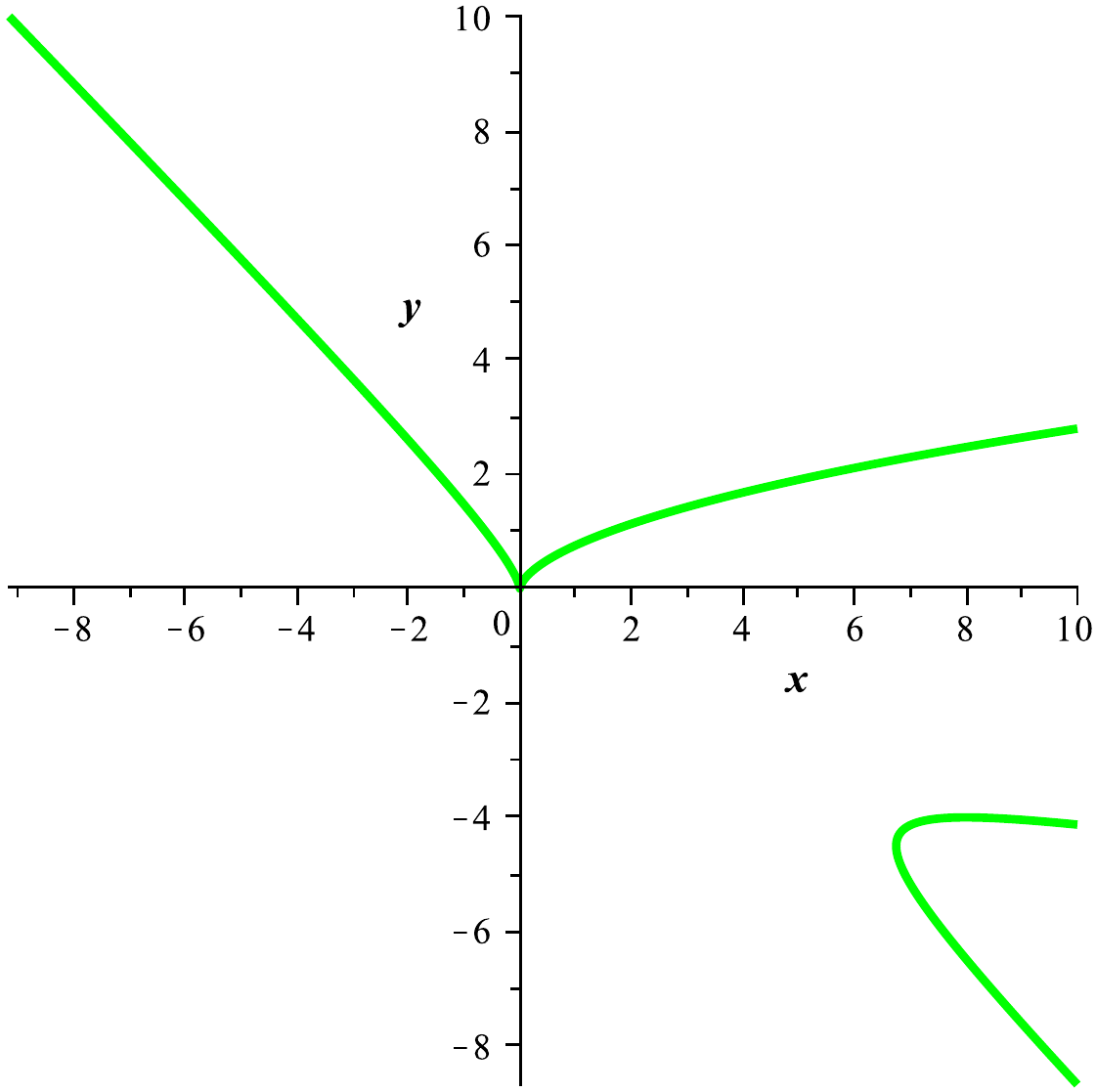}
\caption{\red{${\mathcal X}_1$}: $x^3+y^3-10 x y=0$, \ \blue{${\mathcal X}_2$}: $y^3-x y+1=0$, \
{$\viol{{\mathcal X}_3}$}: $y^3-x^2+x y^2=0$.}
\label{fig-cubics}
\end{figure}
These curves have the following rational parameterizations:
\begin{gather}
\label{eq-alpha}
\red{{\mathcal X}_1}\ \text{is parameterized by}\ \alpha(t)=
\left(\frac{10 t}{t^3+1}, \frac{10 t^2}{t^3+1}\right);
\\
\label{eq-beta}
\blue{{\mathcal X}_2}\ \text{is parameterized by}\ \beta(s)=
\left(\frac{s^3+3 s^2+3 s+2}{s+1}, s+1\right);
\\
\label{eq-gamma}
\viol{{\mathcal X}_3}\ \text{is parameterized by} \ \gamma(w)=
\left(\frac{w^3}{w+1},\frac{w^2}{w+1}\right).
\end{gather}
The projective signature of $\red{{\mathcal X}_1}$ is parameterized by invariants
\begin{gather}
\label{eq-sig-alpha}
K_{\mathcal P}|_{\alpha}(t)=-\frac{9261}{50} \frac{(t^6-t^3+1)^3 t^3}{(t^3-1)^8},
\qquad
T_{\mathcal P}|_{\alpha}(t)=-\frac{21}{10} \frac{(t^3+1)^4}{(t^3-1)^4},
\end{gather}
while the signature of $\blue{{\mathcal X}_2}$ is parameterized by invariants
\begin{gather}
K_{\mathcal P}|_{\beta}(s)
=-\frac{9261}{50} \frac{(s+1)^3 (s^6+6 s^5+15 s^4+19 s^3+12 s^2+3 s+1)^3}{(s^2+3 s+3)^8 s^8},\nonumber
\\
T_{\mathcal P}|_\beta(s)=-\frac{21}{10} \frac{(s^3+3 s^2+3 s+2)^4}{(s^2+3 s+3)^4 s^4}.\label{eq-sig-beta}
\end{gather}
Although it is not obvious, the curves defined by parameterizations~\eqref{eq-sig-alpha}
and~\eqref{eq-sig-beta} satisfy the same implicit equation
\begin{gather}
62523502209+39697461720 \tau-6401203200 \varkappa+5250987000 \tau^2-2032128000 \varkappa \tau
\nonumber\\
\qquad{} +163840000 \varkappa^2+259308000 \tau^3+53760000 \varkappa \tau^2+4410000 \tau^4=0.\label{eq-alpha-beta-sige}
\end{gather}
This is a~sufficient condition for the equality of signatures ${\mathcal S}_{{\mathcal X}_1}$
and ${\mathcal S}_{{\mathcal X}_2}$ over complex numbers, but it is not sufficient over reals.
We can look for a~\emph{real} rational reparameterization $t=\phi(s)$ by solving a~system of two
equations $K_{\mathcal P}|_\alpha(t)=K_{\mathcal P}|_\beta(s)$ and $T_{\mathcal
P}|_\alpha(t)=T_{\mathcal P}|_\beta(s)$ for $t$ in terms of $s$.
One can check that $t=s+1$ provides a~desired reparameterization.
Thus ${\mathcal S}_{{\mathcal X}_1}={\mathcal S}_{{\mathcal X}_2}$ and hence, by
Theorem~\ref{th-main-sig}
\begin{gather*}
\red{{\mathcal X}_1}\underset{{\mathcal{PGL}}(3)}{\cong}\blue{{\mathcal X}_2}.
\end{gather*}
Reparameterization $t=s+1$ allows us to find pairs of points on ${\mathcal X}_1$ and ${\mathcal
X}_2$ which can be transformed to each other by ${\mathcal{PGL}}(3)$-transformation that brings
${\mathcal X}_1$ to ${\mathcal X}_2$.
Since four of such pairs in generic position uniquely determines a~transformation we can compute
that $\blue{{\mathcal X}_2}$ can be transformed to $\red{{\mathcal X}_1}$ by a~transformation
\begin{gather}
\label{eq-t21}
\bar x=\frac{10 y}x,
\qquad
\bar y=\frac{10}x.
\end{gather}

It turns out that cubic $\viol{{\mathcal X}_3}$ has constant ${\mathcal{PGL}}(3)$-invariants
\begin{gather}
\label{eq-sig-gamma}
K_{\mathcal P}|_\gamma(w)=\frac{250047}{12800}\qquad \text{and} \qquad T_{\mathcal P}|_\gamma(w)=0
\end{gather}
and therefore its signature degenerates to a~point.
Thus, by Theorem~\ref{th-main-sig}, $\viol{{\mathcal X}_3}$ is not ${\mathcal{PGL}}(3)$-equivalent
to either \red{${\mathcal X}_1$} or \blue{${\mathcal X}_2$}.

To underscore the difference between the solution of the equivalence problems over real and over
complex numbers, we will consider one more cubic $\brown{{\mathcal X}_4}$ pictured on
Fig.~\ref{fig-delta}, whose rational parameterization is given by
\begin{gather*}
\delta(u)=\left(\frac{u^3+1}{u},
\frac{u^2+1}{u}\right).
\end{gather*}
\begin{figure}[t]
\centering \includegraphics[scale=.25]{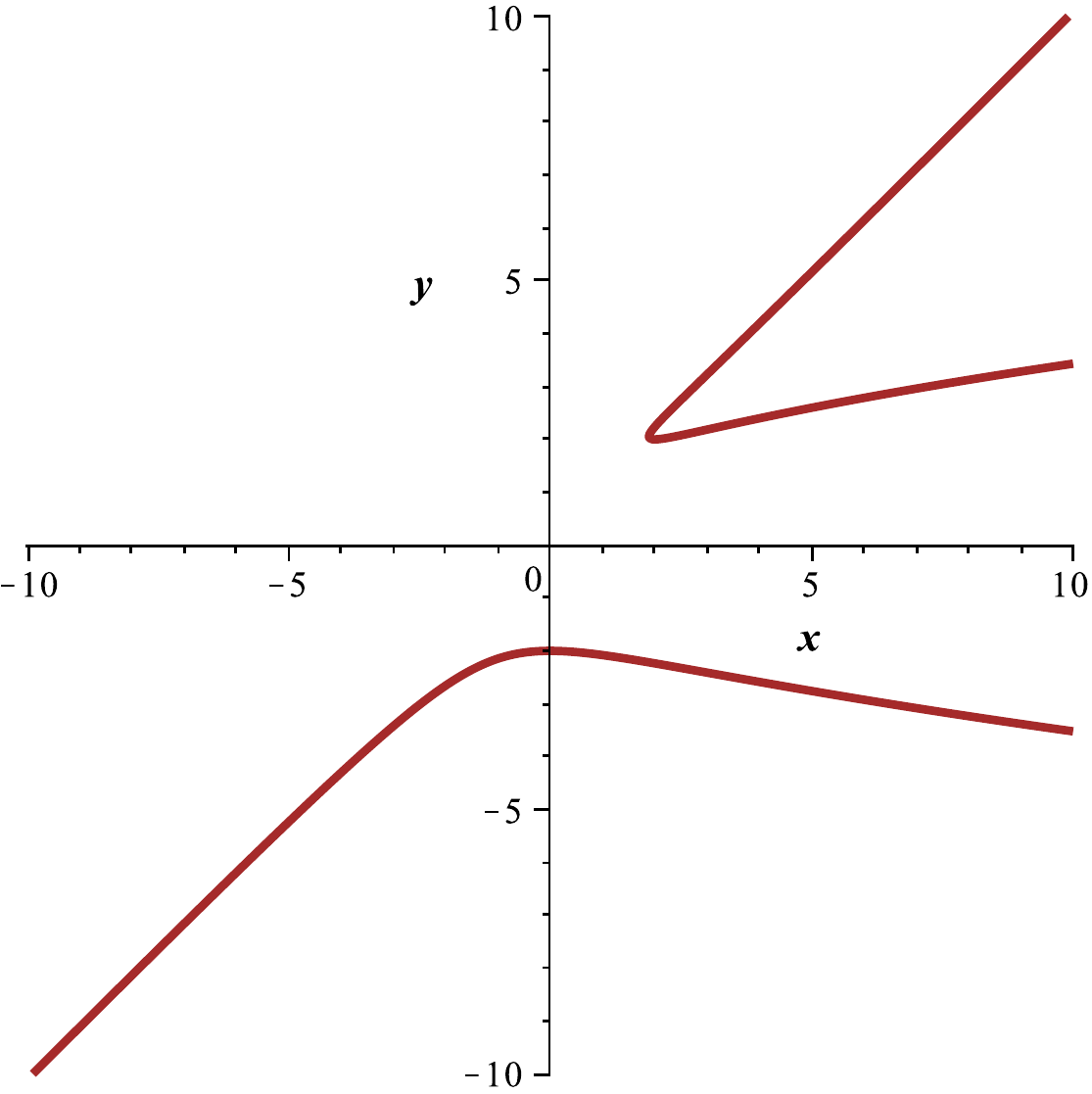}
\caption{\brown{${\mathcal X}_4$}: $y^3-x y^2+x^2-x y+2 x-3 y+2=0$.}
\label{fig-delta}
\end{figure}

The signature of $\brown{{\mathcal X}_4}$ is parameterized by invariants
\begin{gather*}
K_{\mathcal P}|_{\delta}(u)
=-\frac{6751269}{50} \frac{(u-1)^3 (u^2-u+1)^3 (u^3-3 u^2+1)^3 u^3}{(u^3-6 u^2+3 u+1)^8},
\\
T_{\mathcal P}|_\delta(u)=-\frac{189}{10} \frac{(u^3+1-3 u)^4}{(u^3-6 u^2+3 u+1)^4}.
\end{gather*}
Invariants $K_{\mathcal P}|_{\delta}$ and $T_{\mathcal P}|_{\delta}$ satisfy the implicit
equation~\eqref{eq-alpha-beta-sige}.
Since the signatures of~$\red{{\mathcal X}_1}$ and~$\blue{{\mathcal X}_2}$ satisfy the same
implicit equation, we can conclude that, over \emph{the complex numbers}, $\brown{{\mathcal X}_4}$
is projectively equivalent to both $\red{{\mathcal X}_1}$ and $\blue{{\mathcal X}_2}$.
In fact, we can find that for
\begin{gather*}
r=\left(\frac12-\frac12\sqrt{3} i\right)^{1/3}
\end{gather*}
(where we are free to choose any of the three cubic roots) the complex projective transformation
\begin{gather*}
\bar x=
\frac{-6 i \sqrt{3} r y+3 r^2 \big(i \sqrt{3}-3\big)}{2 x-2 i \sqrt{3} r y+r^2 \big(i \sqrt{3}-3\big)},
\qquad
\bar y=
\frac{2 x-r \big(5 i \sqrt{3}+3\big) y+r^2 \big(i \sqrt{3}-9\big)}{2 x-2 i \sqrt{3} r y+r^2 \big(i \sqrt{3}-3\big)}
\end{gather*}
transforms $\brown{{\mathcal X}_4}$ to $\blue{{\mathcal X}_2}$.
Our attempt to solve $T_{\mathcal P}|_\delta(u)=T_{\mathcal P}|_\beta(s)$ and $T_{\mathcal
P}|_\delta(u)=T_{\mathcal P}|_\beta(s)$ for~$u$ in terms of~$s$ gives a~rather involved rational
complex reparameterization that transforms signature map of $\delta$ into the signature map of~$\beta$, but no real reparameterization.
Therefore
\begin{gather*}
\brown{{\mathcal X}_4}\underset{{{\mathcal{PGL}}(3,\mathbb{C})}}\cong\blue{{\mathcal X}_2}
\qquad \text{but} \qquad \brown{{\mathcal X}_4}\underset{{{\mathcal{PGL}}(3,\mathbb{R})}}\ncong\blue{{\mathcal X}_2}.
\end{gather*}
\end{example}

\begin{example}[${\mathcal A}(2)$-equivalence problems] We can again consider three
cubics pictured on Fig.~\ref{fig-cubics}, but now ask if they are ${\mathcal A}(2)$-equivalent.
Recalling that ${\mathcal A}(2)$ is a~subgroup of ${\mathcal{PGL}}(3)$ we can immediately conclude
from the previous example that $\viol{{\mathcal X}_3}$ is not ${\mathcal A}(2)$-equivalent to
\mbox{either}~\red{${\mathcal X}_1$} or~\blue{${\mathcal X}_2$}.
To resolve the equivalence problem for \red{${\mathcal X}_1$} and \blue{${\mathcal X}_2$} we need
to compute their \emph{affine} signatures.
The affine signatures of $\red{{\mathcal X}_1}$ is parameterized by invariants
\begin{gather*}
K_{\mathcal A}|_{\alpha}(t)=-\frac{1}{20} \frac{(t^3-1)^2}{t^3},
\qquad
T_{\mathcal A}|_{\alpha}(t)=-\frac{1}{10},
\end{gather*}
It turns out, that restrictions of both invariants, $K_{\mathcal A}|_{\beta}(s)$ and $T_{\mathcal
A}|_{\beta}(s)$, are non-constant functions of~$s$.
Hence, \red{${\mathcal X}_1$} and \blue{${\mathcal X}_2$} have different affine signatures.
Therefore
\begin{gather*}
\red{{\mathcal X}_1}\underset{{\mathcal A}(2)}\ncong\blue{{\mathcal X}_2}.
\end{gather*}
In fact, affine signatures for all four curves $\red{{\mathcal X}_1}$, $\blue{{\mathcal X}_2}$,
$\viol{{\mathcal X}_3}$ and $\brown{{\mathcal X}_4}$ have different implicit equations and
therefore no two of them are affine equivalent neither over real numbers, nor over complex
numbers.
\end{example}

\section{Algorithm and examples}
\label{algorithms}

The algorithms for solving projection problems are based on a~combination of the projection
criteria of Section~\ref{criteria} and the group equivalence criterion of Section~\ref{group-equiv}.

\subsection{Central projections} The following algorithm is based on the central projection
criterion stated in Theorem~\ref{main-finite-camera} and the group-equivalence criterion stated
in Theorem~\ref{th-main-sig}.
In the algorithm, we compute restrictions of differential functions $\Delta_2$, $K_{\mathcal P}$
and $T_{\mathcal P}$ to a~curve parameterized by $\gamma(t)$ and to a~family of curves
parameterized by $\epsilon(c,s)$, where $c=(c_1,c_2,c_3)$ determines a~member of the family and $s$
serves to parameterize a~curve in the family.
These restrictions are computed by substitution of~\eqref{eq-ykt} into formula~\eqref{eq-D2} for~$\Delta_2$ and into~\eqref{eq-kp} and~\eqref{eq-tp} for $K_{\mathcal P}$ and $T_{\mathcal P}$,
respectively.
When the restrictions to~$\epsilon(c,s)$ are computed, derivatives in~\eqref{eq-ykt} are taken with
respect to~$s$.

One can use general real quantifier elimination packages, such as {\sc Reduce} package in {\sc
Mathe\-ma\-tica}, to perform the steps involving real quantifier elimination problems.
To make an efficient implementation, one needs to take into account specifics of the problems
at hand.
This lies outside of the scope of the current paper and is a~subject of our future work.
\begin{algorithm}[central projections]\label{alg-finite}\qquad

{\sf Input:} Parameterizations
$\Gamma=\big(z_1,z_2,z_3)\in\mathbb{Q}(s)^3$ and $\gamma=\left(x,\,y\right)\in\mathbb{Q}(t)^2$ of
rational algebraic curves ${\mathcal Z}\subset\mathbb{R}^3$ and ${\mathcal X}\subset\mathbb{R}^2$,
respectively, such that ${\mathcal Z}$ is not a~line, that is, $\dot\Gamma\times\ddot\Gamma\neq0$.

{\sf Output:} The truth of the statement:
\begin{gather*}
\exists\, [P]\in\fp, \quad \text{such that} \ {\mathcal X}=\overline{P({\mathcal Z})}.
\end{gather*}

 {\sf Steps}:
\begin{enumerate}\itemsep=0pt
\item\label{st-line}
[{\sf If ${\mathcal X}$ is a~line, then determine whether ${\mathcal Z}$ is coplanar.}]
\\
 if $\left|
\begin{matrix}\dot\gamma
\\
\ddot\gamma
\end{matrix}
\right|\underset{\mathbb{Q}(t)}=0$, then return the truth of the statement $\left|
\begin{matrix}\dot\Gamma
\\
\ddot\Gamma
\\
\dddot\Gamma
\end{matrix}
\right|\underset{\mathbb{Q}(s)}=0.$
\item
\label{st-epsilon}
[{\sf Describe a~family of parametric curves where $c=(c_1,c_2,c_3)$ specifies a~member of the family.}]
\\
$\epsilon:=\left(\frac{z_1+c_1}{z_3+c_3}, \frac{z_2+c_2}{z_3+c_3}\right)\in\mathbb{Q}(c_1,c_2,c_3,s)^2$.
\item
\label{st-computeD}
[{\sf Compute the denominators of the rational invariants.}]
\\
compute $\Delta_2|_\gamma$ and $\Delta_2|_{\epsilon}$, using the formula~\eqref{D2} in Appendix~\ref{appendix}.
\item
\label{st-conic}
[{\sf If ${\mathcal X}$ is a~conic, then determine whether, for some $c$, the rational map
$\epsilon(c,s)$ parameterizes a~conic.}]
\\
if $\Delta_2|_\gamma\underset{\mathbb{Q}(t)}=0$, then return the truth of the statement
\begin{gather}
\exists\, c\in\mathbb{R}^3
\qquad
z_3+c_3\underset{\mathbb{R}(s)}\neq0\,\wedge\,\left|
\begin{matrix}\dot\epsilon
\\
\ddot\epsilon
\end{matrix}
\right|\underset{\mathbb{R}(s)}\neq0\,\wedge\,\Delta_2\big |_{\epsilon}\underset{\mathbb{R}(s)}=0.
\end{gather}
\item
\label{st-notconic}
[{\sf If ${\mathcal Z}$ cannot be projected to a~curve of degree greater than 2, then return {\sc
false.}}]
\\
if $\Delta_2|_\epsilon\underset{\mathbb{R}(c,s)}=0$, then return {\sc false}.
\item
\label{st-computeKT}
[{\sf Compute the rational invariants.}]
\\
compute $K_{\mathcal P}|_\gamma$, $K_{\mathcal P}|_{\epsilon}$, $T_{\mathcal P}|_\gamma$ and $T_{\mathcal
P}|_{\epsilon}$ using the formulae~\eqref{eq-kp},~\eqref{eq-tp} in Appendix~\ref{appendix}.

\item
\label{st-gen}
[{\sf Determine whether, for some $c$, the signature of the Zariski closure $\tilde{\mathcal Z}_c$
of the curve parameterized by $\epsilon(c,s)$ equals to the signature of ${\mathcal X}$.}]
\\
if $K_{\mathcal P}|_\gamma$ is a~constant rational function, then return the truth of the statement
\begin{gather}
\label{st4c1}
\exists\,c\in\mathbb{R}^3:
\quad
c~\text{is generic} \ \wedge K_{\mathcal P}|_{\epsilon}\underset{\mathbb{R}(s)}=
K_{\mathcal P}|_\gamma,
\end{gather}
else return the truth of the statement
\begin{gather}
\nonumber
\exists\,c\in\mathbb{R}^3:
\quad
c~\text{is generic} \ \wedge{K}_{\mathcal P}|_{\epsilon}\text{ is not a constant rational function},
\\ \label{st4c2}
\wedge
\
\forall\, s\in\mathbb{R}
\quad
\Delta_2|_{\epsilon}\underset{\mathbb{R}}\neq0
\quad
\Rightarrow
\quad
\exists\, t\in\mathbb{R}
\quad
K_{\mathcal P}|_{\epsilon}\underset{\mathbb{R}}=
K_{\mathcal P}|_\gamma\,\wedge\,T_{\mathcal P}|_{\epsilon}\underset{\mathbb{R}}=
T_{\mathcal P}|_\gamma,
\end{gather}
where we define
\begin{gather}
\label{eq-gen}
\left[c~\text{is generic}\right]: \ \Longleftrightarrow\ \left[ z_3+c_3\underset{\mathbb{R}(s)}\neq0\,\wedge\,\left|
\begin{matrix}\dot\epsilon
\\
\ddot\epsilon
\end{matrix}
\right|\underset{\mathbb{R}(s)}\neq0\,\wedge\,\Delta_2\big|_{\epsilon}\underset{\mathbb{R}(s)}\neq0\right].
\end{gather}
\end{enumerate} \end{algorithm}

\begin{proof}[Proof of Algorithm~\ref{alg-finite}.]
On the first step of Algorithm~\ref{alg-finite}, we consider the case when~${\mathcal X}$ is
a~line.
Then~${\mathcal Z}$ can be projected to~${\mathcal X}$ if and only if~${\mathcal Z}$ is coplanar.
Both conditions can be checked by computing determinants of certain matrices.
If ${\mathcal X}$ is not a~line we define, on Step~\ref{st-epsilon}, a~rational map $\epsilon$
that parameterizes a~family of curves.
On Step~\ref{st-computeD}, we compute restrictions of differential function~$\Delta_2$ to
$\gamma(t)$ and $\epsilon(c,s)$.
We remind the reader, that in the latter case derivatives are taken with respect to~$s$.
Since on this step we know that ${\mathcal X}$ is not a~line and there are values of~$c$ for which
$\tilde Z_c$ is not a~line, these restrictions are defined.
On Step~\ref{st-conic}, we consider the case when~${\mathcal X}$ is a~conic.
Equivalently, by Corollary~\ref{cor-delta}, $\Delta_2|_\gamma\underset{\mathbb{Q}(t)}=0$.
Then ${\mathcal Z}$ can be projected to~${\mathcal X}$ if and only if~$\exists\, c$ such that the
Zariski closure $\tilde Z_c$ of the curve parameterized by $\epsilon(c,s)$ is projectively
equivalent to ${\mathcal X}$ and therefore is a~conic (see Remark~\ref{rem-excep}).
Equivalently, $\Delta_2|_{\epsilon}\underset{\mathbb{R}(s)}=0$.
If~${\mathcal X}$ is not a~conic, we reach Step~\ref{st-notconic}, where we check for a~possibility
that~$\tilde Z_c$ is a~conic for all parameters~$c$ (equivalently
$\Delta_2|_\epsilon\underset{\mathbb{R}(c,s)}=0$) and therefore it can not be projected to~${\mathcal X}$, which at this step is known to have a~higher degree.
If that is not the case, we proceed to Step~\ref{st-computeKT}, where we compute restrictions of
differential invariants~$K_{\mathcal P}$ and~$T_{\mathcal P}$ to~$\gamma(t)$ and~$\epsilon(c,s)$.
Since on this step we know that~${\mathcal X}$ is of degree greater than~2 and there are values of~$c$ for which~$\tilde Z_c$ is of degree greater than~2, these restrictions are defined.

On Step~\ref{st-gen}, where we know that ${\mathcal X}$ is non-exceptional and decide if there
exists $c\in\mathbb{R}^3$ such that: (1) $\tilde Z_c$ is non-exceptional which is equivalent to
condition~\eqref{eq-gen}; (2) the signatures of the algebraic curves ${\mathcal X}$ and $\tilde
Z_c$ are the same.
On Step~\ref{st-computeKT}, we computed rational functions $K_{\mathcal P}|_{\epsilon}(c,s)$ and
$T_{\mathcal P}|_\epsilon(c,s)$.
We need to show that if we substitute a~specific value $c=c_0\in\mathbb{R}^3$ into these
functions we obtain the same rational functions of $s$ as we would obtain by computing restrictions
of the invariants to the curve parameterized by $\epsilon(c_0,s)$.
For generic values of $c$ defined by~\eqref{eq-gen}, we can show that this is true.
Indeed, it is well known that taking derivatives with respect to one of the variables and
specialization of other variables are commutative operations.
From condition~\eqref{eq-gen} it follows that the curve $\tilde Z_c$ is not a~line and so the
denominators of~\eqref{eq-ykt} are not annihilated by such specialization.
Therefore the restriction of jet variables to a~curve parameterized by $\epsilon(c,s)$ commutes
with a~specialization of $c$.
From condition~\eqref{eq-gen} it also follows that the denominators of~\eqref{eq-kp}
and~\eqref{eq-tp} are not annihilated by a~generic specialization.
Therefore, for a~generic $c_0$, rational functions $K_{\mathcal P}|_{\epsilon}(c_0,s)$ and
$T_{\mathcal P}|_\epsilon(c_0,s)$ equal to the restrictions of the invariants $K_{\mathcal P}$ and
$T_{\mathcal P}$ to $\epsilon(c_0,s)$.
To decide equality of signatures of the algebraic curve parameterized by $\gamma(t)$ and the curve
$\tilde Z_c$ we use Corollary~\ref{cor-cases} with~\eqref{st4c1} analyzing the case of constant
invariant $K_{\mathcal P}|_{\mathcal X}$ and~\eqref{st4c2} the case of non-constant\linebreak
invariant~$K_{\mathcal P}|_{\mathcal X}$.
\end{proof}
\begin{remark}[reconstruction]
If the output is {\sc true}, then, in many cases, in addition to establishing the existence of
$c_1$, $c_2$, $c_3$ in Step~\ref{st-conic} or~\ref{st-gen} of the Algorithm~\ref{alg-finite}, we can
find at least one of such triplets explicitly.
We then know that ${\mathcal Z}$ can be projected to ${\mathcal X}$ by a~projection centered at
$(-c_1,-c_2,-c_3)$.
We can also, in many cases, determine explicitly a~transformation $[A]\in{\mathcal{PGL}}(3)$ that
maps~${\mathcal X}$ to the Zariski closure $\tilde Z_c$ of the image of the map $\epsilon(c,s)$.
We then know that~${\mathcal Z}$ can be projected to~${\mathcal X}$ by the projection
$[P]=[A][\stpf][B]$, where~$\stpf$ and~$B$ are defined by~\eqref{eq-fPB}.
\end{remark}

\begin{example}
\label{ex-cp-cubics}
We would like to decide if the spatial curve ${\mathcal Z}$, pictured on Fig.~\ref{fig-twisted}
and parameteri\-zed~by{\samepage
\begin{gather}
\label{eq-twisted}
\Gamma(s)=\big(s^3, s^2, s\big),\qquad s\in\mathbb{R}
\end{gather}
projects to any of the four cubics planar cubics described in Example~\ref{ex-psig}.}

\begin{figure}[t] \centering
\includegraphics[scale=.40]{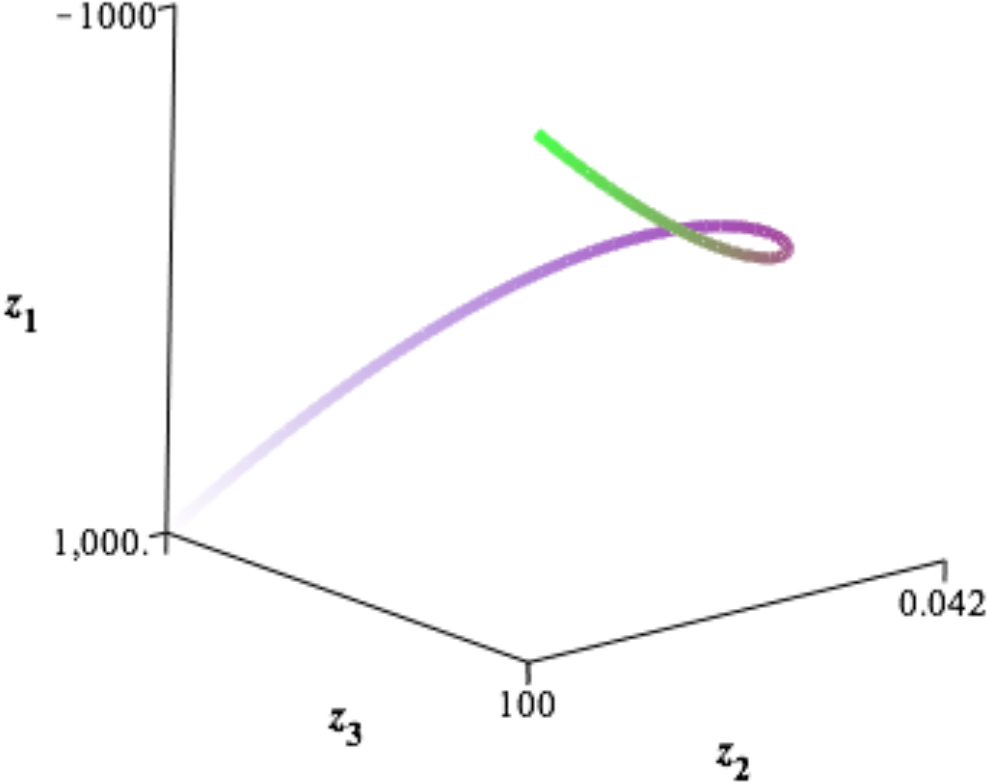}
\caption{Twisted cubic.}
\label{fig-twisted}
\end{figure}

We start with cubics $\red{{\mathcal X}_1}$, $\blue{{\mathcal X}_2}$ pictured on
Fig.~\ref{fig-cubics}, whose parameterizations are given by~\eqref{eq-alpha}
and~\eqref{eq-beta}, respectively.
Since these two cubics are ${\mathcal{PGL}}(3)$-equivalent then the twisted cubic can be either
projected to both of them or to none of them (see Proposition~\ref{proj-classes}).
Let us ``run'' the algorithm for $\blue{{\mathcal X}_2}$.
Following Algorithm~\ref{alg-finite}, since $\blue{{\mathcal X}_2}$ is not a~line, we proceed to
Step~\ref{st-epsilon} and define a~family of curves
\begin{gather}
\label{eq-epsilont}
\epsilon(c_1,c_2,c_3,s)=\left(\frac{s^3+c_1}{s+c_3}, \frac{s^2+c_2}{s+c_3}\right).
\end{gather}
Since $\blue{{\mathcal X}_2}$ is not a~conic, we proceed to Step~\ref{st-computeKT}.
Invariants for ${\mathcal X}_1$ and ${\mathcal X}_2$ are given by formulae~\eqref{eq-sig-alpha}
and~\eqref{eq-sig-beta}, respectively.
The formulae for rational functions $K_{\mathcal P}|_{\epsilon}(c_1,c_2,c_3,s)$ and $T_{\mathcal
P}|_{\epsilon}(c_1,c_2,c_3,s)$ are too long to be included in the paper, but can be computed using
our {\sc Maple} code~\cite{maple-code}.
Step~\ref{st-gen} of Algorithm~\ref{alg-finite} returns {\sc true}.
In fact, $c=(1,0,0)$ satisfies conditions~\eqref{st4c2} and therefore ${\mathcal Z}$ can be
projected to both $\red{{\mathcal X}_1}$ and $\blue{{\mathcal X}_2}$ with a~camera centered at
$(-1,0,0)$.
We note that the canonical projection
\begin{gather}
\label{eq-p1}
x=\frac{z_1+1}{z_2},\qquad y=\frac{z_3}{z_2}.
\end{gather}
with this center maps ${\mathcal Z}$ to $\blue{{\mathcal X}_2}$ (although resulting
parameterization differs from~\eqref{eq-beta}.) From decomposition~\eqref{eq-fp-dec} we know that
a~projection from ${\mathcal Z}$ to $\red{{\mathcal X}_1}$ is the composition of the
projection~\eqref{eq-p1} and ${\mathcal{PGL}}(3)$ transformation~\eqref{eq-t21} that brings
$\blue{{\mathcal X}_2}$ to $\red{{\mathcal X}_1}$.
The resulting projection is $x=\frac{10 z_3}{z_1+1}, y=\frac{10 z_2}{z_1+1}$.

As a~side remark, we note $K_{\mathcal P}|_{\epsilon}(c,s)$ and $T_{\mathcal P}|_{\epsilon}(c,s)$
satisfy the implicit equation\eqref{eq-alpha-beta-sige} independently of~$c$.
Thus, the closure of the signature $\tilde{\mathcal Z}_c$ belongs to the closure of the signature
of $\blue{{\mathcal X}_2}$ (and~$\red{{\mathcal X}_1}$).
Therefore, we could immediately establish existence of a~complex projection, because the
conditions~\eqref{eq-gen} for $c$ being non-exceptional and such that $\dot{K}_{\mathcal
P}|_{\epsilon}\underset{\mathbb{C}(s)}\neq0$ define a~Zariski open non-empty subset of~$\mathbb{C}^3$.

Let us now consider $\brown{{\mathcal X}_4}$ pictured on Fig.~\ref{fig-delta}.
Again Algorithm~\ref{alg-finite} reaches Step~\ref{st-gen} and returns {\sc true}.
In fact, $c=(1,1,0)$ satisfies conditions~\eqref{st4c2} and therefore ${\mathcal Z}$ can be
projected to $\brown{{\mathcal X}_4}$ with a~camera centered at $(-1,-1,0)$.
Indeed, the canonical projection with this center{\samepage
\begin{gather*}
x=\frac{z_1+1}{z_3},\qquad y=\frac{z_2+1}{z_3}
\end{gather*}
maps ${\mathcal Z}$ to ${\mathcal X}_4$.}

This is a~good point to compare real and complex projection problems.
As was established in Example~\ref{ex-psig}, the signatures of all three curves $\red{{\mathcal
X}_1}$, $\blue{{\mathcal X}_2}$ and $\brown{{\mathcal X}_4}$ have the same implicit
equation~\eqref{eq-alpha-beta-sige}.
Therefore, they are all ${{\mathcal{PGL}}(3,\mathbb{C})}$-equivalent.
Therefore, since there is a~projection from~${\mathcal Z}$ to~$\red{{\mathcal X}_1}$ and
$\blue{{\mathcal X}_2}$, centered at $(-1,0,0)$, there is a~\emph{complex} projection centered at
$(-1,0,0)$ from~${\mathcal Z}$ to~$\brown{{\mathcal X}_4}$.
We also established, in Example~\ref{ex-psig} that
\begin{gather}
\brown{{\mathcal X}_4}\underset{{{\mathcal{PGL}}(3,\mathbb{R})}}\ncong\blue{{\mathcal X}_2}
\end{gather}
and, therefore, there
is no \emph{real} projection centered at $(-1,0,0)$ from ${\mathcal Z}$ to $\brown{{\mathcal X}_4}$
which, as we have seen, does not preclude the existence of a~\emph{real} projection with
a~different center $(-1,-1,0)$.

Finally we consider $\viol{{\mathcal X}_3}$, pictured on Fig.~\ref{fig-cubics} with
parameterization~\eqref{eq-gamma}.
From Example~\ref{ex-psig} we know that invariants for $\viol{{\mathcal X}_3}$ are constants, see~\eqref{eq-sig-gamma}.
Following Algorithm~\ref{alg-finite}, we need to decide whether there
exists $c\in\mathbb{R}^3$, such that $\epsilon(c,s)$ does not parameterize a~line or a~conic and
\begin{gather*}
K_{\mathcal P}|_\epsilon(c,s)=\frac{250047}{12800}
\qquad
\forall\, s\in\mathbb{R}.
\end{gather*}
This is, indeed, true for $c_1=c_2=0$ and $c_3=1$.
This is sufficient to conclude the existence of a~real projection.
We can check that ${\mathcal Z}$ can be projected to $\viol{{\mathcal X}_3}$ by the a~central
projection $x=\frac{z_1}{z_3+1}$, $y=\frac{z_2}{z_3+1}$.
\end{example}

The above example underscores Remark~\ref{rem-pr-gr}: although the twisted cubic can
be projected to each of the planar curves $\red{{\mathcal X}_1}$, $\blue{{\mathcal X}_2}$,
$\viol{{\mathcal X}_3}$ and $\brown{{\mathcal X}_4}$, the planar curve $\viol{{\mathcal X}_3}$ is
{\em not} ${\mathcal{PGL}}(3,\mathbb{C})$-equivalent to $\red{{\mathcal X}_1}$, or $\blue{{\mathcal
X}_2}$, or $\brown{{\mathcal X}_4}$.
Also $\brown{{\mathcal X}_4}$ is not ${\mathcal{PGL}}(3,\mathbb{R})$-equivalent to $\red{{\mathcal
X}_1}$ or $\blue{{\mathcal X}_2}$.

\begin{example}
\label{ex-cp-cubix-conic}
In this example, we establish that the twisted cubic~\eqref{eq-twisted} can be projected to any
conic.
As before, the family of rational curves $\epsilon$ is defined by~\eqref{eq-epsilont}.
The twisted cubic can be projected to a~conic if and only if there exists $c$ such that
$\epsilon(c,s)$ does not parameterize a~line and $\Delta_2|_{\epsilon}\underset{\mathbb{C}(s)}=0$.
We can easily check that $\epsilon(c,s)$ is not a~line for all $c$ and that
$\Delta_2|_{\epsilon}\underset{\mathbb{C}(s)}=0$ whenever
\begin{gather}
\label{eq-twconic}
c_1=c_3^3,\qquad c_2=-c_3^2.
\end{gather}
Let $c_3=-a$, where $a$ is any real number.
From~\eqref{eq-twconic} it follows that $\epsilon$ parameterizes a~conic if and only if
$c=(-a^3,-a^2,-a)$.
The corresponding canonical projection centered $(a^3,a^2,a)$ maps the twisted cubic to the
parabola
\begin{gather*}
y^2+a y-x+a^2=0.
\end{gather*}
Since all conics are ${\mathcal{PGL}}(3)$-equivalent, we established that the twisted cubic can be
projected to any conic.
Moreover, we established that the twisted cubic is projected to a~conic if and only if the center
of the projection lies on the twisted cubic.
\end{example}

So far in all our examples the outcome of the projection algorithm was {\sc true}.
Below is an example with {\sc false} outcome.
\begin{example} We will show that the twisted cubic~\eqref{eq-twisted} can not be projected to the
quintic $\omega(t)=(t,t^5)$.
The signature of the quintic is parameterized by a~constant map
\begin{gather*}
K_{\mathcal P}|_{\omega}(t)=\frac{1029}{128}\qquad \text{and} \qquad T_{\mathcal P}|_{\omega}(t)=0,
\qquad
\forall\, t.
\end{gather*}
Following Algorithm~\ref{alg-finite}, we need to decide whether there exists~$c\in\mathbb{R}^3$,
such that $\epsilon(c,s)$ does not parameterize a~line or a~conic and
\begin{gather*}
K_{\mathcal P}|_{\epsilon}(c,s)=\frac{1029}{128},\qquad \forall\, s\in\mathbb{R}.
\end{gather*}
Substitution of several values of~$s$ in the above equation yields a~system of polynomial equations
for $c_1,c_2,c_3\in\mathbb{R}$ that has no solutions.
We conclude that there is no central projection from~${\mathcal Z}$ to~$\omega(t)=(t,t^5)$.
This outcome is, of course, expected, because a~cubic can not be projected to a~curve of degree
higher than~3.
\end{example}

\subsection{Parallel projections}

The algorithm for parallel projections is based on the reduced parallel projection criterion stated
in Corollary~\ref{reduced-aff-camera}.
This algorithm follows the same logic but has more steps than Algorithm~\ref{alg-finite}, because
we need to decide whether a~given planar curve is ${\mathcal A}(2)$-equivalent to a~curve
parameterized by $\alpha(s)=(z_2(s),z_3(s))$, or to a~curve parameterized by
$\beta(b,s)=\left({z_1(s)+b z_2(s)}, z_3(s)\right)$ for some $b\in\mathbb{R}$, or to a~curve
parameterized by $\delta(a_1,a_2,s)=\left({z_1(s)+a_1z_3(s)},z_2+a_2z_3(s)\right)$ for some
$a=(a_1,a_2)\in\mathbb{R}^2$.
Since the affine transformations are considered, projective invariants are replaced with
affine invariants (see~\eqref{eq-kta}).
Due to its similarity to Algorithm~\ref{alg-finite}, we refrain from writing out the steps of the
parallel projection algorithm and content ourselves with presenting examples.
A {\sc Maple} implementation of the parallel projection algorithm over complex numbers is included
in~\cite{maple-code}.

\begin{example}
As a~follow-up to Example~\ref{ex-cp-cubics}, it is natural to ask whether the twisted cubic can be
projected to any of the cubics considered in that example by a~\emph{parallel projection}.
Our implementation of the parallel projection algorithms~\cite{maple-code} provides a~negative
answer to this question, the twisted cubic \emph{can not} be projected to $\red{{\mathcal X}_1}$,
or $\blue{{\mathcal X}_2}$ or $\viol{{\mathcal X}_3}$, or $\brown{{\mathcal X}_4}$ under a~parallel
projection even over complex numbers.
There are plenty of rational cubics to which the twisted cubic can be projected by a~parallel
projection.
For example, the orthogonal projection to the $z_1,z_3$-plane projects the twisted cubic to
$(s^3,s)$.
\end{example}

\begin{example}
As a~follow-up to Example~\ref{ex-cp-cubix-conic}, we consider the problem of the parallel
projection of the twisted cubic to a~conic.
To answer this question we first consider a~curve $\alpha(s)=(z_2(s),z_3(s))=(s^2,s)$, which is
a~parabola, and therefore the twisted cubic can be projected to any parabola.
We then define a~one-parametric family of curves
\begin{gather*}
\beta(b,s)=\big({z_1(s)+b z_2(s)}, z_3(s)\big)=\big(s^3+b s^2,s\big)
\end{gather*}
and a~two-parametric family
\begin{gather*}
\delta(a,s)=\big({z_1(s)+a_1z_3(s)},z_2+a_2z_3(s)\big)=\big({s^3+a_1 s},s^2+a_2 s\big).
\end{gather*}
Obviously, there are no parameters $a$ or $b$ such that a~curve in those two families becomes
a~conic (we can check this formally by computing rational functions~$\Delta_2|_\alpha$ and
$\Delta_2|_\beta$ and seeing that there are no values of~$a$ or~$b$ that will make them zero
functions of~$s$).
Thus we conclude that the twisted cubic can not be projected to either a~hyperbola or an ellipse by
a~parallel projection.
\end{example}

For a~less obvious example and to finally get away from the twisted cubic we
consider the following:
\begin{example}
We would like to decide whether the spatial curve ${\mathcal Z}$
parameterized by
\begin{gather*}
\Gamma(s)=\big(s^4, s^2, s\big),\qquad s\in\mathbb{R}
\end{gather*}
can be projected to ${\mathcal X}$ parameterized by
\begin{gather*}
\gamma(t)=\big(t^4+t^3,t^2+t\big),\qquad t\in\mathbb{R}.
\end{gather*}

\begin{figure}[t] \centering
\begin{minipage}[b]{0.45\linewidth} \centering
\includegraphics[scale=.32]{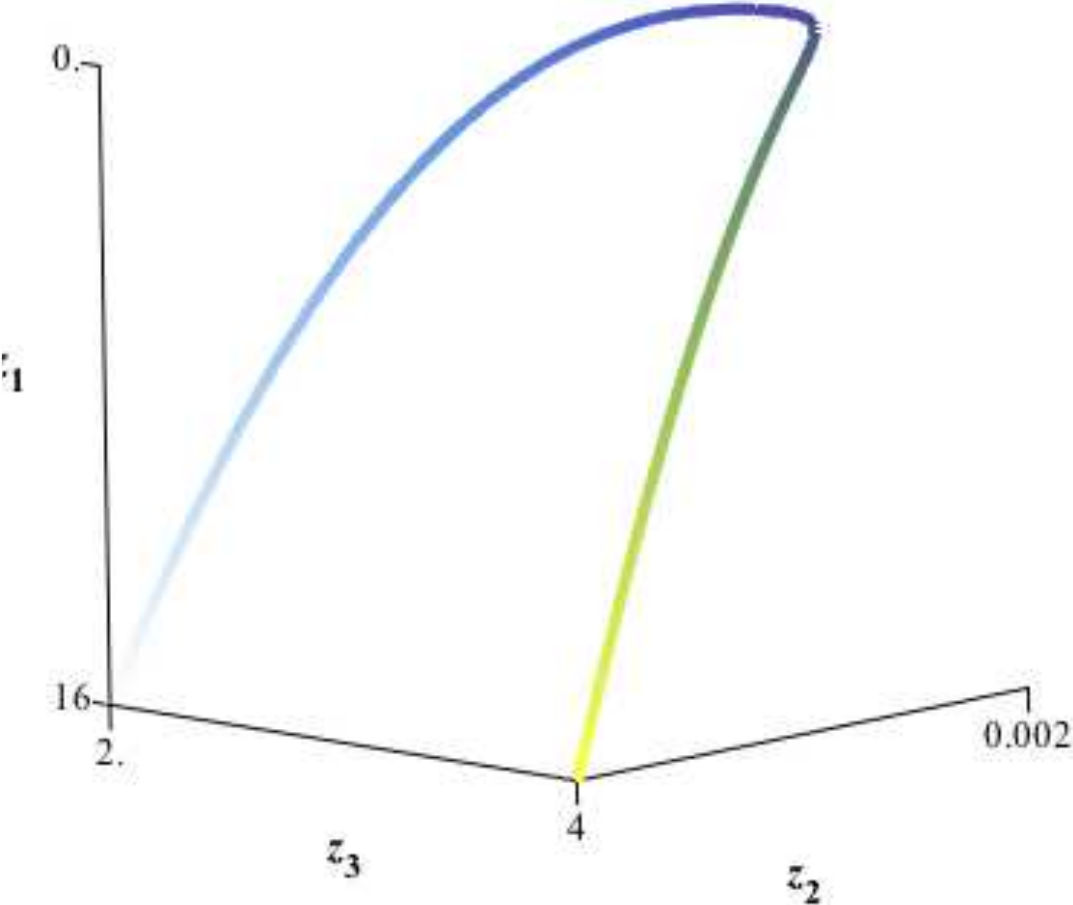} \caption{${\mathcal Z}$:
$\Gamma(s)=\big(s^4, s^2, s\big)$.}
\end{minipage} \begin{minipage}[b]{0.45\linewidth}
\centering \includegraphics[scale=.26]{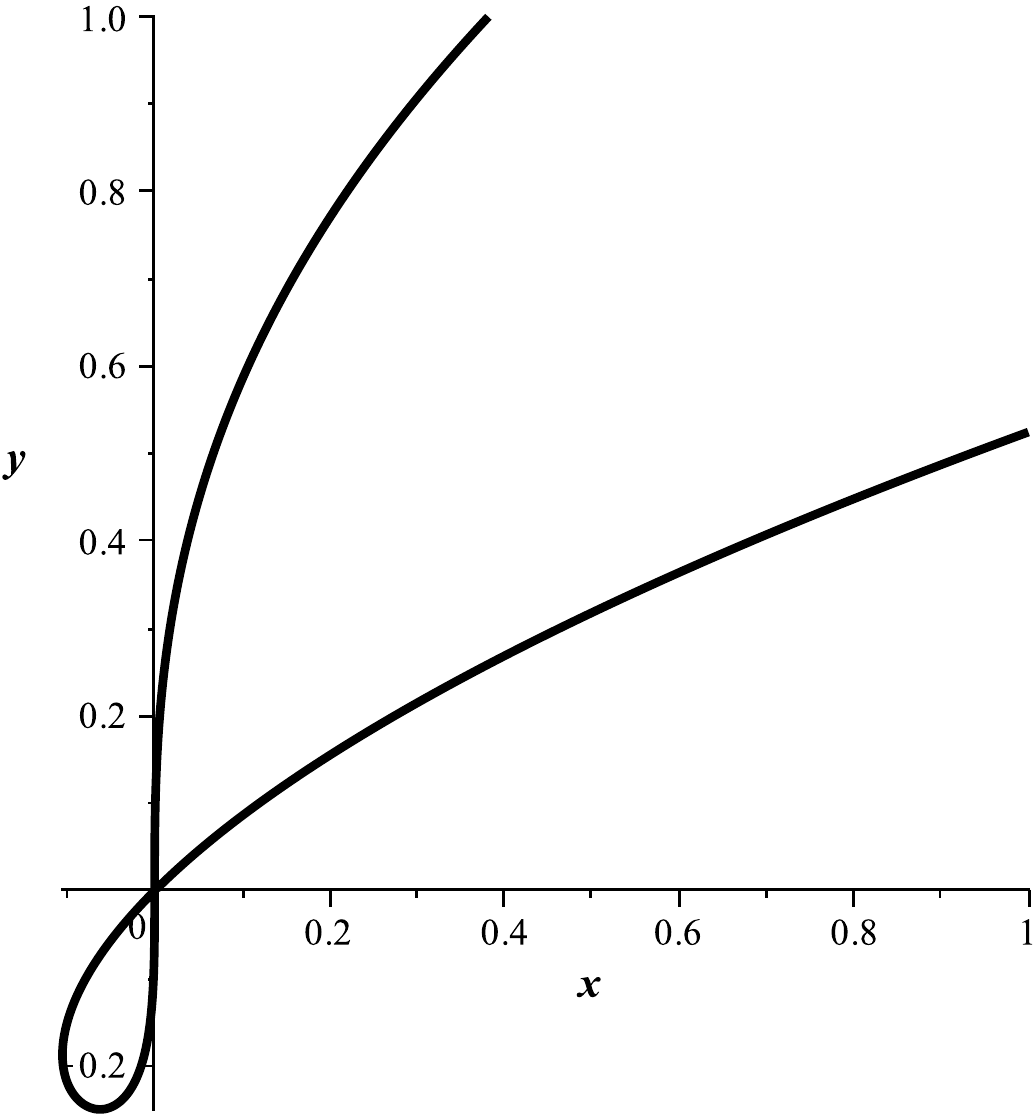}
\caption{{${\mathcal X}$}: $y^4-2 x y^2+x^2-x y=0$.}
\end{minipage}
\end{figure}

The signature of ${\mathcal X}$ is parameterized by invariants
\begin{gather*}
K_\mathcal{A}|_\gamma(t)
=-1600 \frac{(24 t^5+51 t^4+57 t^3+33 t^2+9 t+1)^2}{(24 t^3+32 t^2+24 t+5)^3},
\\
T_\mathcal{A}|_\gamma(t)
=-40 \frac{448 t^7+1304 t^6+1956 t^5+1735 t^4+915 t^3+287 t^2+51 t+4}{(24 t^3+32 t^2+24 t+5)^2}.
\end{gather*}
The curve $\alpha(s)=(z_2(s),z_3(s))=(s^2,s)$ is a~parabola and is not ${\mathcal A}(2)$-
equivalent to ${\mathcal X}$.
A curve from the family
\begin{gather*}
\beta(b,s)=\big({z_1(s)+b z_2(s)}, z_3(s)\big)=\big(s^4+b s^2,s\big)
\end{gather*}
is not ${\mathcal A}(2)$-exceptional and its invariants are given by
\begin{gather*}
K_\mathcal{A}|_\beta=100\,\frac{(3\,b-14\,s^2)^2\,s^2}{(b-14\,s^2)^3},
\qquad
T_\mathcal{A}|_\beta=-5\,\frac{b^2-56\,b\,s^2+140\,s^4}{(b-14\,s^2)^2}.
\end{gather*}
Independently of the value of $b$ all curves in the family have the same signature equation
\begin{gather*}
245\,\tau^3+40000-448\,\varkappa^2+3780\,\tau\,\varkappa-1575\,\tau^2+14525\,\varkappa-6000\,\tau=0,
\end{gather*}
which is different from the implicit equation for the signature for ${\mathcal X}$, and therefore
the curves from this family are not ${\mathcal A}(2)$-equivalent to ${\mathcal X}$.
We finally consider a~two-parametric family
\begin{gather*}
\delta(a,s)=\big({z_1(s)+a_1\,z_3(s)},z_2+a_2\,z_3(s)\big)=\left(s^4+a_1\,s,\,s^2+a_2\,s\right)
\end{gather*}
and find out that for $a_1=20$ and $a_2=2$ the implicit equations of the signatures of ${\mathcal
X}$ and the curve parameterized by $\delta(a,s)$ are the same.
Thus we conclude that ${\mathcal Z}$ projects to ${\mathcal X}$ by a~parallel projection over the
complex numbers.

To check that this remains true over $\mathbb{R}$, we look at the invariants of $\delta(20,2,s)$
\begin{gather*}
K_\mathcal{A}|_\delta(20,2,s)
=-\frac{25}2 \frac{(6 s^5+21 s^4+84 s^3+90 s^2+30 s+25)^2}{(3 s^3+7 s^2+25 s+5)^3},
\\
T_\mathcal{A}|_\delta(20,2,s)
=-5 \frac{14 s^7+65 s^6+294 s^5+535 s^4+450 s^3+415 s^2+250 s+25}{(3 s^3+7 s^2+25 s+5)^2}.
\end{gather*}
By solving equations $K_\mathcal{A}|_\delta(20,2,s)=K_\mathcal{A}|_\gamma(t)$ and
$T_\mathcal{A}|_\delta(20,2,s)=T_\mathcal{A}|_\gamma(t)$ for $s$ in terms of $t$, we find a~real
reparameterization $s=1+4 t$ which matches the signature maps of two curves, i.e.\
${S}|_\gamma(t)={S}|_\delta(20,2,1+4\,t)$.
Thus the signatures of ${\mathcal X}$ and the curve parameterized by $\delta(20,2,s)$ are identical
and, therefore, ${\mathcal Z}$ projects to ${\mathcal X}$ by a~parallel projection over the real
numbers.

We proceed to find a~projection.
Since ${S}|_\gamma(t)={S}|_\delta(20,2,1+4 t)$, we not only know that the exists ${\mathcal
A}(2)$-transformation $A$ that maps~$\gamma(t)$ to $\delta\big(20,2,(1+4 t)\big)$, but that for
any value of~$t$ the transformation~$A$ maps the point $\gamma(t)$ to the point
$\delta(20,2,(1+4 t)$.
Three pairs of points in general position are sufficient to recover an affine
transformation.
Using three pairs of points corresponding to $t=0,1,2$ we find the affine transformation
\begin{gather*}
\bar x=256 x+96 y+21,
\qquad
\bar y=16 y+3
\end{gather*}
that transforms ${\mathcal X}$ to the curve parameterized by $\delta(20,2,s)$.
From decomposition~\eqref{eq-pp-decomp} we can now recover a~\emph{parallel projection}
\begin{gather*}
x=\frac1{156} z_1-\frac3{128} z_2+\frac1{32} z_3-\frac3{256},
\qquad
y=\frac1{16} z_2+\frac18 z_3-\frac3{16}
\end{gather*}
that maps ${\mathcal Z}$ to ${\mathcal X}$.
As a~side remark, we note that there is no \emph{central projection} (either complex or real) that
maps ${\mathcal Z}$ to ${\mathcal X}$.
\end{example}

\section{Extensions}
\label{sect-ext}
\subsection{Projection problem for non-rational algebraic curves}
The projection criteria of Theorems~\ref{main-finite-camera} and~\ref{main-affine-camera} and
the group equivalence criteria of Theorem~\ref{th-main-sig} are valid for non-rational algebraic
curve.
An implementation of the projection algorithm is more challenging, however, when the curves are
described as zero sets of polynomials, rather than by a~parameterization.
To illustrate these challenges, we can, in fact, consider a~rational curve, the twisted cubic, but
this time we will define this curve by implicit equations.
\begin{example} The twisted cubic is defined as the zero set ${\mathcal Z}$ of the system of two
polynomials $g_1=z_1-z_2 z_3$ and $g_2=z_2-z_3^2$.
Following the projection criterion of Theorem~\ref{main-finite-camera}, we define a~family of
algebraic curves
\begin{gather*}
\tilde{\mathcal Z}_{c_1,c_2,c_3}=\overline{\left\{\left(\frac{z_1+c_1}{z_3+c_3},
 \frac{z_2+c_2}{z_3+c_3}\right)
\Bigg| (z_1,z_2,z_3)\in{\mathcal Z}\right\}}.
\end{gather*}
To restrict differential functions, and in particular invariants, to curves from the family
$\tilde{\mathcal Z}_{c_1,c_2,c_3}$ we need to compute their implicit equations.
A naive approach would be to define an ideal
\begin{gather*}
Y:=\big\langle(z_3+c_3)x-(z_1+c_1),
(z_3+c_3)y-(z_2+c_2)\,{z_3+c_3},
(z_3+c_3)\delta-1,
\\
\phantom{Y:=\big\langle}
g_1(z_1,z_2,z_3),
g_2(z_1,z_2,z_3)\big\rangle
\end{gather*}
and compute the elimination ideal $Y\cap\mathbb{R}[c_1,c_2,c_3,x,y]$.
This leads to a~polynomial (cubic in~$x$ and~$y$)
\begin{gather*}
A_c(x,y)=
\big(c_3^3-c_1\big) y^3+\big(c_3^2+c_2\big) y^2 x-\big(c_3^2+c_2\big) x^2+(c_1+c_3 c_2) x y-3 \big(c_1 c_3+c_3^2 c_2\big) y^2
\\
\phantom{A_c(x,y)=}
+2 \big(c_1 c_3-c_2^2\big) x+3 \big(c_2^2 c_3+c_1 c_2\big) y-c_1^2-c_2^3,
\end{gather*}
which, indeed, defines the curve $\tilde{\mathcal Z}_{c_1,c_2,c_3}$ provided $c_3^3-c_1\neq0$ and
$c_3^2+c_2\neq0$, but not otherwise.
The underlying issue is non-commutativity of specialization of parameters $c$ with the elimination
(or, in more geometric language, non-commutativity of intersection and Zariski closure).

To find implicit equations for $\tilde{\mathcal Z}_{c_1,c_2,c_3}$ for the remaining values of the
parameters $c$, we need a~detailed analysis of the constructible set obtained by the projection of
the variety of the ideal $Y$ onto $(x,y,c_1,c_2,c_3)$-subspace.
This could be done using, for instance, {\sc RegularChains} package in {\sc Maple}.
We find out that when $c_3^3-c_1\neq0$, but $c_3^2+c_2=0$, the curve $\tilde{\mathcal
Z}_{c_1,c_2,c_3}$ is a~zero set of a~cubic
\begin{gather*}
B_c(x,y)=y^3+3 c_3 y^2-x y+3 c_3^2 y-2 c_3 x+c_1+c_3^3.
\end{gather*}
When both $c_3^3-c_1=0$ and $c_3^2+c_2=0$ (that is the case when the center of the projection lies
on the twisted cubic), then $\tilde{\mathcal Z}_{c_1,c_2,c_3}$ is a~zero set of a~quadratic
polynomial
\begin{gather*}
C_c(x,y)=y^2+c_3 y-x+c_3^2.
\end{gather*}
On Step~\ref{st-epsilon} of Algorithm~\ref{alg-finite}, where we are describing the family of
curves $\tilde{\mathcal Z}_{c_1,c_2,c_3}$, we must produce all three possible implicit equations
$A_c(x,y)=0$, when $c_3^3-c_1\neq0$ and $c_3^2+c_2\neq0$, $B_c(x,y)=0$, when $c_3^3-c_1\neq0$, but
$c_3^2+c_2=0$ and $C_c(x,y)=0$, when $c_3^3-c_1=0$ and $c_3^2+c_2=0$.
Then the rest of the algorithm should run for each of these cases with appropriate conditions on
$c$.
\end{example}

We found that, for majority of the examples, producing the set of all possible
implicit equations for the curves $\tilde{\mathcal Z}_{c_1,c_2,c_3}$~\eqref{poset} from the given
implicit equations of an algebraic curve ${\mathcal Z}\subset\mathbb{R}^3$ to be a~very challenging
computational task.
\subsection{Projection problem for finite lists of points}
The projection criterion of Theorem~\ref{main-finite-camera} adapts to finite lists of points
as follows: \begin{theorem}[central projection criterion for finite lists]
\label{main-finite-points}
A list $Z=({\mathbf z}^1,\dots,{\mathbf z}^m)$ of $m$ points in $\mathbb{R}^3$ with coordinates
${\mathbf z}^l=(z^l_1,z^l_2,z^l_3)$, $l=1,\dots, m$, projects onto a~list $X=({\mathbf
x}^1,\dots,{\mathbf x}^m)$ of $m$ points in $\mathbb{R}^2$ with coordinates ${\mathbf
x}^l=(x^l,y^l)$, $l=1,\dots, m$, by a~central projection if and only if there exist
$c_1,c_2,c_3\in\mathbb{R}$ and $[A]\in{\mathcal{PGL}}(3)$, such that
\begin{gather*}
[x^l,y^l,1]^{\tr}=[A][z_1^l+c_1, z_2^l+c_2, z_3^l+c_3]^{\tr}\qquad \text{for} \quad l=1,\dots, m.
\end{gather*}
\end{theorem}

The proof of Theorem~\ref{main-finite-points} is a~straightforward adaptation of
the proof of Theorem~\ref{main-finite-camera}.
The parallel projection criteria for curves, given in Theorem~\ref{main-affine-camera} and
Corollary~\ref{reduced-aff-camera}, are adapted to the finite lists in an analogous way.

The central and the parallel projection problems for lists of $m$ points is therefore reduced to
a~modification of the problems of equivalence of two lists of $m$ points in $\mathbb{PR}^2$ under
the action of ${\mathcal{PGL}}(3)$ and ${\mathcal A}(2)$ groups, respectively.
A separating set of invariants for lists of $m$ points in $\mathbb{PR}^2$ under the ${\mathcal
A}(2)$-action consists of ratios of certain areas and is listed, for instance, in Theorem~3.5
of~\cite{olver01}.
Similarly, a~separating set of invariants for lists of $m$ ordered points in $\mathbb{PR}^2$ under
the ${\mathcal{PGL}}(3)$-action consists of cross-ratios of certain areas and is listed, for instance,
in Theorem~3.10 in~\cite{olver01}.
In the case of central projections we, therefore, obtain a~system of polynomial equations on
$c_1$, $c_2$ and $c_3$ that have solutions if and only if the given set $Z$ projects to the given set
$X$ and an~analog of Algorithm~\ref{alg-finite} follows.
The parallel projections are treated in a~similar way.
Details of this adaptation appear in the dissertation~\cite{burdis10}.

\looseness=-1
We note, however, that there are other computationally efficient solution of the projection
problem for lists of points.
In their book~\cite{hartley04}, Hartley and Zisserman describe algorithms that are based on
straightforward approach: one writes a~system of equations that relates pairs of the corresponding
points in the lists $Z\subset\mathbb{R}^3$ and $X\subset\mathbb{R}^3$ and determines if this system
has a~solution.
The book also describes algorithms for finding parameters of the camera that produces an optimal
(under various criteria) but not exact match between the object and the image.

In~\cite{stiller06, stiller07}, the authors present a~solution to the problem of deciding whether
or not there exists a~parallel projection of a~list $Z=({\mathbf z}^1,\dots,{\mathbf z}^m)$ of $m$
points in $\mathbb{R}^3$ to a~list $X=({\mathbf x}^1,\dots,{\mathbf x}^m)$ of $m$ points in
$\mathbb{R}^2$, without finding a~projection explicitly.
They identify the lists $Z$ and $X$ with the elements of certain Grassmanian spaces and use
Pl\"uker embedding of Grassmanians into projective spaces to explicitly define the algebraic
variety that characterizes pairs of sets related by a~parallel projection.
They also define of an object/image distance between lists of points $Z\subset\mathbb{R}^3$ and
$X\subset\mathbb{R}^2$, such that the distance is zero if and only if there exists a~parallel
projection that maps $Z$ to $X$.
\begin{figure}[t]\centering \includegraphics[scale=0.6]{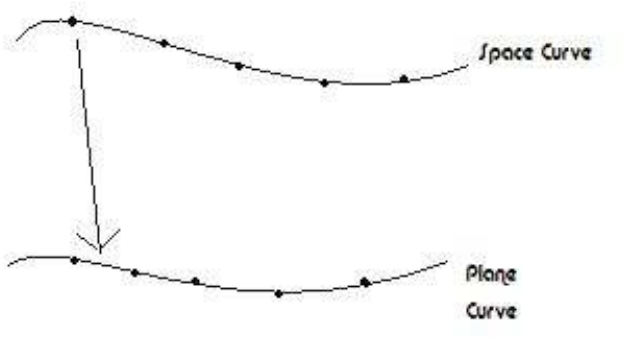}
\caption{Projection
problem for curves vs.~projection problems for lists of points.} \label{fig-dc}
\end{figure}
As illustrated by Fig.~\ref{fig-dc}, a~solution of the projection problem for
lists of points does not provide an immediate solution to the discretization of the projection
problem for curves.
Indeed, let $Z=({\mathbf z}^1,\dots,{\mathbf z}^m)$ be a~discrete sampling of a~spatial curve
${\mathcal Z}$ and $X=({\mathbf x}^1,\dots,{\mathbf x}^m)$ be a~discrete sampling of a~planar curve~${\mathcal X}$.
It might be impossible to project the list~$Z$ onto~$X$, even when the curve~${\mathcal Z}$ can be
projected to the curve~${\mathcal X}$.

\subsection{Applications: challenges and ideas}
A discretization of projection algorithms for curves will pave a~road to real-life applications and is a~topic of our future research.
Such algorithms may utilize invariant numerical approximations of differential invariants
presented in~\cite{boutin00, calabi98}.
Differential invariants and their approxi\-ma\-tions are highly sensitive to image perturbations and,
therefore, pre-smoothing of the data is required to use them.
Since affine and projective invariants involve high order derivatives, this approach may not be
practical.
Other types of invariants, such as semi-differential (or joint) inva\-riants~\mbox{\cite{olver01,vang92-1}}, integral inva\-riants~\mbox{\cite{feng09,hann02, sato97}} and moment
invariants~\cite{taub92,xu06} are less sensitive to image perturbations and may be employed to
solve the group-equivalence problem.

One of the essential contributions of~\cite{stiller07, stiller06} is the definition of an
object/image distance between ordered sets of $m$ points in $\mathbb{R}^3$ and $\mathbb{R}^2$, such
that the distance is zero if and only if these sets are related by a~projection.
Since, in practice, we are given only an approximate position of points, a~``good'' object/image
distance provides a~tool for deciding whether a~given set of points in $\mathbb{R}^2$ is a~good
approximation of a~projection of a~given set of points in $\mathbb{R}^3$.
Defining such object/image distance in the case of curves is an important direction of further
research.

\appendix
\section{Appendix}
\label{appendix}

We provide explicit formulae for invariants in terms of jet coordinates.
For convenience we recall our notation
\begin{gather}
\Delta_1:=3 \jy4 \jy2-5 \big[\jy3\big]^2,
\nonumber\\
\label{D2}
\Delta_2:=9 \jy5 \big[\jy2\big]^2-45 \jy4 \jy3 \jy2+40 \big[\jy3\big]^3.
\end{gather}
A classifying set of ${\mathcal A}(2)$-invariants~\eqref{aff-inv} is given by
\begin{gather}
\label{eq-kta}
K_{\mathcal A}=\frac{(\Delta_2)^2}{(\Delta_1)^3},
\\
\nonumber
T_{\mathcal A}=
\frac{9 \jy6 \big[\jy2\big]^3-63 \jy5 \jy3 \big[\jy2\big]^2-45 \big[\jy4\big]^2 \big[\jy2\big]^2+255 \jy4 \big[\jy3\big]^2 \jy2-160 \big[\jy3\big]^4}
{(\Delta_1)^2},
\end{gather}
while a~classifying set of rational ${\mathcal{PGL}}(3)$-invariants~\eqref{proj-inv} is given by
\begin{gather}
\label{eq-kp}
K_{\mathcal P}=\frac{729}{8 (\Delta_2)^8} \Big[18 \jy7 \big[\jy2\big]^4
\Delta_2-189 \big[\jy6\big]^2 \big[\jy2\big]^6
\\
\nonumber
\phantom{K_{\mathcal P}=}
{}+126 \jy6 \big[\jy2\big]^4
\left(9 \jy5 \jy3 \jy2+15 \big[\jy4\big]^2 \jy2-25 \jy4 \big[\jy3\big]^2\right)
\\
\nonumber
\phantom{K_{\mathcal P}=}
{}-189 \big[\jy5\big]^2 \big[\jy2\big]^4
\left(4 \big[\jy3\big]^2+15 \jy2 \jy4\right)
\\
\nonumber
\phantom{K_{\mathcal P}=}
{}+210 \jy5 \jy3 \big[\jy2\big]^2\left(63 \big[\jy4\big]^2 \big[\jy2\big]^2-60 \jy4 \big[\jy3\big]^2 \jy2+32 \big[\jy3\big]^4\right)
\\
\nonumber
\phantom{K_{\mathcal P}=}
{}-525 \jy4\jy2 \left(9 \big[\jy4\big]^3 \big[\jy2\big]^3+15 \big[\jy4\big]^2 \big[\jy3\big]^2 \big[\jy2\big]^2\right.
\\
\nonumber
\left.\phantom{K_{\mathcal P}=}
{}-60 \jy4 \big[\jy3\big]^4 \jy2+64 \big[\jy3\big]^6\right)+11200 \big[\jy3\big]^8 \Big]^3,
\\
\nonumber
T_{\mathcal P}=\frac{243\big[\jy2\big]^4}{2(\Delta_2)^4} \Big[2\jy8\jy2 (\Delta_2)^2
-8\jy7\Delta_2 \Big(9\jy6\big[\jy2\big]^3-36\jy5 \jy3\big[\jy2\big]^2
\\
\nonumber
\phantom{T_{\mathcal P}=}
{} -45\big[\jy4\big]^2\big[\jy2\big]^2+120\jy4 \big[\jy3\big]^2\jy2-40 \big[\jy3\big]^4\Big)+504 \big[\jy6\big]^3 \big[\jy2\big]^5
\\
\nonumber
\phantom{T_{\mathcal P}=}
{}-504 \big[\jy6\big]^2 \big[\jy2\big]^3 \Big(9 \jy5 \jy3 \jy2+15 \big[\jy4\big]^2 \jy2-25 \jy4 \big[\jy3\big]^2\Big)
\\
\nonumber
\phantom{T_{\mathcal P}=}
{} +28\jy6\Big(432\big[\jy5\big]^2\big[\jy3\big]^2\big[\jy2\big]^3+243\big[\jy5\big]^2\jy4 \big[\jy2\big]^4
-1800 \jy5 \jy4 \big[\jy3\big]^3 \big[\jy2\big]^2 \\
\nonumber
\phantom{T_{\mathcal P}=}
{}
-240\jy5 \big[\jy3\big]^5 \jy2+540\jy5 \big[\jy4\big]^2 \big[\jy3\big] \big[\jy2\big]^3+6600 \big[\jy4\big]^2 \big[\jy3\big]^4 \jy2
\\
\nonumber
\phantom{T_{\mathcal P}=}
{}
-2000
\jy4 \big[\jy3\big]^6-5175 \big[\jy4\big]^3 \big[\jy3\big]^2 \big[\jy2\big]^2+1350 \big[\jy4\big]^4 \big[\jy2\big]^3\Big)
\\
\nonumber
\phantom{T_{\mathcal P}=}
{}-2835 \big[\jy5\big]^4 \big[\jy2\big]^4
+252 \big[\jy5\big]^3\jy3 \big[\jy2\big]^2
\Big(9\jy4 \jy2-136 \big[\jy3\big]^2\Big)
\\
\nonumber
\phantom{T_{\mathcal P}=}
{}-35840 \big[\jy5\big]^2 \big[\jy3\big]^6-630 \big[\jy5\big]^2 \big[\jy4\big] \big[\jy2\big]\Big(69\big[\jy4\big]^2\big[\jy2\big]^2-160\big[\jy3\big]^4
\\
\phantom{T_{\mathcal P}=}
{}-153\jy4\big[\jy3\big]^2\big[\jy2\big]\Big)
+2100 \jy5 \big[\jy4\big]^2 \jy3
\Big(72 \big[\jy3\big]^4+63 \big[\jy4\big]^2 \big[\jy2\big]^2\label{eq-tp}
\\
\nonumber
\phantom{T_{\mathcal P}=}
{}-193\jy4\big[\jy3\big]^2\jy2\Big)\!-7875\big[\jy4\big]^4\Big(8\big[\jy4\big]^2\big[\jy2\big]^2\!-22\jy4 \big[\jy3\big]^2 \big[\jy2\big]+9 \big[\jy3\big]^4\Big)\Big].
\end{gather}

\subsection*{Acknowledgements}

The project was supported in part by NSA grant H98230-11-1-0129.
We would like to thank the referees for careful reading of our manuscript and valuable suggestions.

\pdfbookmark[1]{References}{ref}
\LastPageEnding


\begin{thebibliography}{99}
\footnotesize\itemsep=0pt

\bibitem{stiller07}
Arnold G., Stiller P.F., Mathematical aspects of shape analysis for object
 recognition, in Proceedings of IS\&T/SPIE Joint Symposium ``Visual
 Communications and Image Processing'' (San Jose, CA, 2007), \textit{SPIE
 Proceedings}, Vol.~6508, Editors C.W.~Chen, D.~Schonfeld, J.~Luo, 2007,
 65080E, 11~pages.

\bibitem{stiller06}
Arnold G., Stiller P.F., Sturtz K., Object-image metrics for generalized weak
 perspective projection, in Statistics and Analysis of Shapes, \href{http://dx.doi.org/10.1007/0-8176-4481-4_10}{\textit{Model. Simul.
 Sci. Eng. Technol.}}, Birkh\"auser Boston, Boston, MA, 2006, 253--279.

\bibitem{bix98}
Bix R., Conics and cubics. A~concrete introduction to algebraic curves,
 \textit{Undergraduate Texts in Mathematics}, Springer-Verlag, New York, 1998.

\bibitem{blaschke23}
Blaschke W., Vorlesungen \"uber Differentialgeometrie und geometrische
 Grundlagen von Einsteins Relativit\"atstheorie. II.~Affine
 Differentialgeometrie, J.~Springer, Berlin, 1923.

\bibitem{boutin00}
Boutin M., Numerically invariant signature curves, \href{http://dx.doi.org/10.1023/A:1008139427340}{\textit{Int.~J. Comput.
 Vis.}} \textbf{40} (2000), 235--248, \mbox{\href{http://arxiv.org/abs/math-ph/9903036}{math-ph/9903036}}.

\bibitem{burdis10}
Burdis J.M., Object-image correspondence under projections, Ph.D. thesis,
 North Carolina State University, 2010.

\bibitem{maple-code}
Burdis J.M., Kogan I.A., Supplementary material for ``Object-image correspondence for curves under projections'',
 \url{http://www.math.ncsu.edu/~iakogan/symbolic/projections.html}.

\bibitem{bk12}
Burdis J.M., Kogan I.A., Object-image correspondence for curves under central
 and parallel projections, in \href{http://dx.doi.org/10.1145/2261250.2261306}{Proceedings of the Symposium on Computational
 Geometry} (Chapel Hill, NC, 2012), ACM, New York, 2012, 373--382.

\bibitem{calabi98}
Calabi E., Olver P.J., Shakiban C., Tannenbaum A., Haker S., Differential and
 numerically invariant signature curves applied to object recognition,
 \href{http://dx.doi.org/10.1023/A:1007992709392}{\textit{Int.~J. Comput. Vis.}} \textbf{26} (1998), 107--135.

\bibitem{C37}
Cartan E., La th\'eorie des groupes finis et continus et la ge\'om\'etrie
 diff\'erentielle trait\'ees par la m\'ethode du rep\`ere mobile,
 Gauthier-Villars, Paris, 1937.

\bibitem{caviness-johnson}
Caviness B.F., Johnson J.R. (Editors), Quantifier elimination and cylindrical
 algebraic decomposition, \href{http://dx.doi.org/10.1007/978-3-7091-9459-1}{\textit{Texts and Monographs in Symbolic Computation}},
 Springer-Verlag, Vienna, 1998.

\bibitem{CLO96}
Cox D., Little J., O'Shea D., Ideals, varieties, and algorithms. An
 introduction to computational algebraic geometry and commutative algebra, 2nd
 ed., \textit{Undergraduate Texts in Mathematics}, Springer-Verlag, New York, 1997.

\bibitem{faugeras94}
Faugeras O., Cartan's moving frame method and its application to the geometry
 and evolution of curves in the Euclidean, affine and projective planes,
 in Application of Invariance in Computer Vision, \href{http://dx.doi.org/10.1007/3-540-58240-1_2}{\textit{Springer-Verlag
 Lecture Notes in Computer Science}}, Vol.~825, Editors J.L.~Mundy,
 A.~Zisserman, D.~Forsyth, Springer-Verlag, Berlin, 1994, 9--46.

\bibitem{faugeras01}
Faugeras O., Luong Q.T., The geometry of multiple images. The laws that govern
 the formation of multiple images of a scene and some of their applications,
 MIT Press, Cambridge, MA, 2001.

\bibitem{feldmar95}
Feldmar J., Ayache N., Betting F., 3D-2D projective registration of free-form
 curves and surfaces, in Proceedings of the Fifth International Conference on
 Computer Vision (ICCV'95), \href{http://dx.doi.org/10.1109/ICCV.1995.466891}{IEEE Computer Society}, Washington, 1995, 549--556.

\bibitem{feng09}
Feng S., Kogan I., Krim H., Classification of curves in 2{D} and 3{D} via
 affine integral signatures, \href{http://dx.doi.org/10.1007/s10440-008-9353-9}{\textit{Acta Appl. Math.}} \textbf{109} (2010),
 903--937, \href{http://arxiv.org/abs/0806.1984}{arXiv:0806.1984}.

\bibitem{fulton}
Fulton W., Algebraic curves. An introduction to algebraic geometry, \textit{Advanced
 Book Classics}, Addison-Wesley Publishing Company, Redwood City, CA, 1989.

\bibitem{Gug63}
Guggenheimer H.W., Differential geometry, McGraw-Hill, New York, 1963.

\bibitem{hann02}
Hann C.E., Hickman M.S., Projective curvature and integral invariants,
 \href{http://dx.doi.org/10.1023/A:1020617228313}{\textit{Acta Appl. Math.}} \textbf{74} (2002), 177--193.

\bibitem{hartley04}
Hartley R., Zisserman A., Multiple view geometry in computer vision, Cambridge
 University Press, Cambridge, 2001.

\bibitem{ho12}
Hoff D., Olver P.J., Extensions of invariant signatures for object
 recognition, \href{http://dx.doi.org/10.1007/s10851-012-0358-7}{\textit{J.~Math. Imaging Vision}} \textbf{45} (2013), 176--185.

\bibitem{hong93}
Hong H. (Editor), Special issue on computational quantifier elimination, \textit{Comput.~J.}
\textbf{36} (1993).

\bibitem{hk:focm}
Hubert E., Kogan I.A., Smooth and algebraic invariants of a group action: local
 and global constructions, \href{http://dx.doi.org/10.1007/s10208-006-0219-0}{\textit{Found. Comput. Math.}} \textbf{7} (2007),
 455--493.

\bibitem{kogan03}
Kogan I.A., Two algorithms for a moving frame construction, \href{http://dx.doi.org/10.4153/CJM-2003-013-2}{\textit{Canad.~J.
 Math.}} \textbf{55} (2003), 266--291.

\bibitem{musso09}
Musso E., Nicolodi L., Invariant signatures of closed planar curves,
 \href{http://dx.doi.org/10.1007/s10851-009-0155-0}{\textit{J.~Math. Imaging Vision}} \textbf{35} (2009), 68--85.

\bibitem{olver:yellow}
Olver P.J., Applications of {L}ie groups to differential equations,
\textit{Graduate Texts in Mathematics}, Vol.~107, 2nd ed., Springer-Verlag,
 New York, 1993.

\bibitem{olver01}
Olver P.J., Joint invariant signatures, \textit{Found. Comput. Math.}
 \textbf{1} (2001), 3--67.

\bibitem{vinberg89}
Popov V.L., Vinberg E.B., Invariant theory, in Algebraic geometry. {IV}.~Linear
 algebraic groups. Invariant theory, \textit{Encyclopaedia of Mathematical
 Sciences}, Vol.~55, Editors A.N.~Parshin, I.R.~Shafarevich, Springer-Verlag,
 Berlin, 1994, 122--278.

\bibitem{sato97}
Sato J., Cipolla R., Affine integral invariants for extracting symmetry
 axes, \href{http://dx.doi.org/10.1016/S0262-8856(97)00011-5}{\textit{Image Vision Comput.}} \textbf{15} (1997), 627--635.

\bibitem{tarski:51}
Tarski A., A decision method for elementary algebra and geometry, 2nd ed.,
 University of California Press, Berkeley, 1951.

\bibitem{taub92}
Taubin G., Cooper D.B., Object recognition based on moment (or algebraic)
 invariants, in Geometric Invariance in Computer Vision, Editors J.L.~Mundy,
 A.~Zisserman, \textit{Artificial Intelligence}, MIT Press, Cambridge, MA, 1992,
 375--397.

\bibitem{vang92-1}
Van~Gool L.J., Moons T., Pauwels E., Oosterlinck A., Semi-differential
 invariants, in Geometric Invariance in Computer Vision, Editors J.L.~Mundy,
 A.~Zisserman, \textit{Artificial Intelligence}, MIT Press, Cambridge, MA, 1992,
 157--192.

\bibitem{xu06}
Xu D., Li H., 3-D affine moment invariants generated by geometric
 primitives, in Proceedings of 18th International Conference on Pattern
 Recognition, Vol.~2, \href{http://dx.doi.org/10.1109/ICPR.2006.21}{IEEE Computer Society}, Washington, 2008, 544--547.

\end{thebibliography}
\end{document}